\documentclass[12pt, twoside]{article}
\usepackage{amsmath,amsthm,amssymb}
\usepackage{times}
\usepackage{enumerate}

\def\titlerunning#1{\gdef\titrun{#1}}
\makeatletter
\def\author#1{\gdef\autrun{\def\and{\unskip, }#1}\gdef\@author{#1}}
\def\address#1{{\def\and{\\\hspace*{18pt}}\renewcommand{\thefootnote}{}%
\footnote {#1}}%

\markboth{L. Brandolini, C. Choirat, L. Colzani, G. Gigante, R. Seri, G. Travaglini}{\titrun}}

\makeatother
\def\email#1{e-mail: #1}
\def\subjclass#1{{\renewcommand{\thefootnote}{}%
\footnote{\emph{Mathematics Subject Classification (2010):} #1}}}
\def\keywords#1{\par\medskip
\noindent\textbf{Keywords.} #1}

\newtheorem{thm}{Theorem}[section]
\newtheorem{cor}[thm]{Corollary}
\newtheorem{lem}[thm]{Lemma}

\theoremstyle{definition}
\newtheorem{defin}[thm]{Definition}
\newtheorem{rem}[thm]{Remark}

\begin{document}

%%%%% To ease editing, add:

\baselineskip=17pt

%%%%%%%%%%%%%%%%

%% In the running head, give an abbreviation of the title.
\titlerunning{Quadrature rules and distribution of points on manifolds}

\title{Quadrature rules and distribution of points on manifolds}

\author{Luca Brandolini
\and
Christine Choirat
\and
Leonardo Colzani
\and
Giacomo Gigante*
\and
Raffaello Seri
\and
Giancarlo Travaglini}

\date{}

\maketitle

\address{L. Brandolini: Dipartimento di Ingegneria dell'Informazione e Metodi Matematici,
Universit\`{a} di Bergamo, Viale Marconi 5, 24044 Dalmine, Bergamo, Italia; \email{luca.brandolini@unibg.it}
\and
C. Choirat: Department of Economics, School of Economics and Business Management, Universidad de Navarra, Edificio de Bibliotecas (Entrada Este), 31080 Pamplona (Spain); \email{cchoirat@unav.es}
\and
L. Colzani: Dipartimento di Matematica e Applicazioni, Edificio U5,
Universit\`{a} di Milano Bicocca,
Via R.Cozzi 53, 20125 Milano, Italia; \email{leonardo.colzani@unimib.it}
\and
G. Gigante (corresponding author): Dipartimento di Ingegneria dell'Informazione e Metodi Matematici,
Universit\`{a} di Bergamo,
Viale Marconi 5, 24044 Dalmine, Bergamo, Italia; \email{giacomo.gigante@unibg.it}
\and
R. Seri: Universit\`{a} degli Studi dell'Insubria,
Dipartimento di Economia,
via Monte Generoso 71, 21100 Varese, Italia; \email{raffaello.seri@uninsubria.it}
\and
G. Travaglini: Dipartimento di Statistica, Edificio U7,
Universit\`{a} di Milano-Bicocca,
Via Bicocca degli Arcimboldi 8, 20126 Milano, Italia; \email{giancarlo.travaglini@unimib.it}}


%%%%%%%%

\begin{abstract}
We study the error in quadrature rules on a compact manifold,
\begin{equation*}
\left\vert \sum_{j=1}^{N}\omega _{j}f\left( z_{j}\right) -\int_{\mathcal{M}%
}f(x)dx\right\vert \leq c\mathcal{D}\left\{ z_{j}\right\} \mathcal{V}\left\{
f\right\} .
\end{equation*}
As in the Koksma Hlawka inequality, $\mathcal{D}\left\{ z_{j}\right\} $\ is a sort of
discrepancy of the sampling points and $\mathcal{V}\left\{ f\right\} $\ is a
generalized variation of the function. In particular, we give sharp
quantitative estimates for quadrature rules of functions in Sobolev classes.

%% Keywords are optional
\end{abstract}

\section{Introduction}
In what follows, $\mathcal{M}$ is a smooth compact $d$ dimensional
Riemannian manifold with Riemannian measure $dx$, normalized so that the
total volume of the manifold is 1, and $\Delta $ is the Laplace Beltrami
operator. This operator is self-adjoint in $\mathbb{L}^{2}(\mathcal{M})$, it
has a sequence of eigenvalues $\left\{ \lambda ^{2}\right\} $ and an
orthonormal complete system of eigenfunctions $\left\{ \varphi _{\lambda
}(x)\right\} $, $\Delta \varphi _{\lambda }(x)=\lambda ^{2}\varphi _{\lambda
}(x)$. The eigenvalues, possibly repeated, can be ordered with increasing
modulus. In particular, the first eigenvalue is 0 and the associated
eigenfunction is 1. An example is the torus $\mathbb{T}^{d}=\mathbb{R}^{d}/%
\mathbb{Z}^{d}$ with the Laplace operator $-\sum \partial ^{2}/\partial
x_{j}^{2}$, eigenvalues $\left\{ 4\pi ^{2}\left\vert k\right\vert
^{2}\right\} _{k\in \mathbb{Z}^{d}}$ and eigenfunctions $\left\{ \exp \left(
2\pi ikx\right) \right\} _{k\in \mathbb{Z}^{d}}$. Another example is the
sphere $\mathbb{S}^{d}=\left\{ x\in \mathbb{R}^{d+1},\;\left\vert
x\right\vert =1\right\} $ with normalized surface measure and with the
angular component of the Laplacian in the space $\mathbb{R}^{d+1}$,
eigenvalues $\left\{ n(n+d-1)\right\} _{n=0}^{+\infty }$ and eigenfunctions
the restriction to the sphere of homogeneous harmonic polynomials in space.

A classical problem is to approximate an integral $\int_{\mathcal{M}}f(x)dx$
with Riemann sums $N^{-1}\sum_{j=1}^{N}f\left( z_{j}\right) $, or weighted
analogues $\sum_{j=1}^{N}\omega _{j}f\left( z_{j}\right) $, and what
follows will be concerned with the discrepancy between integrals and sums
for functions in Sobolev classes $\mathbb{W}^{\alpha ,p}\left( \mathcal{M}%
\right) $ with $1\leq p\leq +\infty $ and $\alpha >d/p$. The assumption $%
\alpha >d/p$ guarantees the boundedness and continuity of the function $%
f\left( x\right) $, otherwise $f\left( z_{j}\right) $ may be not defined. As
a motivation, assume there exists a decomposition of $\mathcal{M}$\ into $N$%
\ disjoint pieces $\mathcal{M}=U_{1}\cup U_{2}\cup ...\cup U_{N}$\ and these
pieces have measures $N^{-1}$\ and diameters at most $cN^{-1/d}$. Choosing a
point $z_{j}$ in each $U_{j}$, one obtains the estimate
\begin{gather*}
\left\vert N^{-1}\sum_{j=1}^{N}f\left( z_{j}\right) -\int_{\mathcal{M}%
}f(x)dx\right\vert \\
\leq \sum_{j=1}^{N}\int_{U_{j}}\left\vert f\left( z_{j}\right)
-f(x)\right\vert dx\leq \sup_{\left\vert y-x\right\vert \leq
cN^{-1/d}}\left\{ \left\vert f\left( y\right) -f(x)\right\vert \right\} .
\end{gather*}

In particular, since functions in $\mathbb{W}^{\alpha ,p}\left( \mathcal{M}%
\right) $ with $\alpha >d/p$ are H\"{o}lder continuous of degree $\alpha
-d/p $, one obtains
\begin{equation*}
\left| N^{-1}\sum_{j=1}^{N}f\left( z_{j}\right) -\int_{\mathcal{M}%
}f(x)dx\right| \leq cN^{-\left( \alpha -d/p\right) /d}\left\| f\right\| _{%
\mathbb{W}^{\alpha ,p}\left( \mathcal{M}\right) }.
\end{equation*}

On the other hand, it will be shown that suitable choices of the sampling
points $\left\{ z_{j}\right\} $ improve the exponent $1/p-\alpha /d$ to $%
-\alpha /d$ and this is best possible. More precisely, the main results in
this paper are the following:

\textit{(1) For every }$d/2<\alpha <d/2+1$ \textit{there exists }$c>0$%
\textit{\ such that if }$\mathcal{M}=U_{1}\cup U_{2}\cup ...\cup U_{N}$%
\textit{\ is a decomposition of the manifold in disjoint pieces with measure
}$\left| U_{j}\right| =\omega _{j}$\textit{, then there exists a
distribution of points }$\left\{ z_{j}\right\} _{j=1}^{N}$\textit{\ with }$%
z_{j}\in U_{j}$\textit{\ such that for every function }$f(x)$\textit{\ in
the Sobolev space }$\mathbb{W}^{\alpha ,2}\left( \mathcal{M}\right) $,
\begin{equation*}
\left| \sum_{j=1}^{N}\omega _{j}f\left( z_{j}\right) -\int_{\mathcal{M}%
}f(x)dx\right| \leq c\max_{1\leq j\leq N}\left\{ \mathrm{diameter}\left(
U_{j}\right) ^{\alpha }\right\} \left\| f\right\| _{\mathbb{W}^{\alpha ,2}}.
\end{equation*}

\textit{(2) Assume that the points }$\left\{ z_{j}\right\} _{j=1}^{N}$%
\textit{\ and the positive weights }$\left\{ \omega _{j}\right\} _{j=1}^{N}$%
\textit{\ give an exact quadrature for all eigenfunctions with eigenvalues }$%
\lambda ^{2}<r^{2}$\textit{, that is}
\begin{equation*}
\sum_{j=1}^{N}\omega _{j}\varphi _{\lambda }\left( z_{j}\right) =\int_{%
\mathcal{M}}\varphi _{\lambda }(x)dx=\left\{
\begin{array}{l}
1\;\;\;\text{\textit{if }}\lambda =0\text{\textit{,}} \\
0\;\;\;\text{if }0<\lambda <r\text{\textit{.}}%
\end{array}
\right.
\end{equation*}

\textit{Then for every }$1\leq p\leq +\infty $\textit{\ and }$\alpha >d/p$\
\textit{there exist }$c>0$\textit{, which may depend on $\mathcal{M}$}, $p$%
\textit{, }$\alpha $\textit{, but is independent of }$r$\textit{, }$\left\{
z_{j}\right\} _{j=1}^{N}$\textit{\ and }$\left\{ \omega _{j}\right\}
_{j=1}^{N}$\textit{, such that}
\begin{equation*}
\left| \sum_{j=1}^{N}\omega _{j}f\left( z_{j}\right) -\int_{\mathcal{M}%
}f(x)dx\right| \leq cr^{-\alpha }\left\| f\right\| _{\mathbb{W}^{\alpha ,p}}.
\end{equation*}

\textit{(3) If }$1\leq p\leq +\infty $\textit{\ and }$\alpha >d/p$\textit{,
then there exist }$c>0$\textit{\ and sequences of points }$\left\{
z_{j}\right\} _{j=1}^{N}$\textit{\ and positive weights }$\left\{ \omega
_{j}\right\} _{j=1}^{N}$\textit{\ with}
\begin{equation*}
\left| \sum_{j=1}^{N}\omega _{j}f\left( z_{j}\right) -\int_{\mathcal{M}%
}f(x)dx\right| \leq cN^{-\alpha /d}\left\| f\right\| _{\mathbb{W}^{\alpha
,p}}.
\end{equation*}

\textit{(4) For every }$1\leq p\leq +\infty $\textit{\ and }$\alpha >d/p$%
\textit{\ there exists }$c>0$\textit{\ such that for every distribution of
points }$\left\{ z_{j}\right\} _{j=1}^{N}$\textit{\ and numbers }$\left\{
\omega _{j}\right\} _{j=1}^{N}$\textit{\ there exists a function }$f(x)$%
\textit{\ in }$\mathbb{W}^{\alpha ,p}\left( \mathcal{M}\right) $\textit{\
with}
\begin{equation*}
\left| \sum_{j=1}^{N}\omega _{j}f\left( z_{j}\right) -\int_{\mathcal{M}%
}f(x)dx\right| \geq cN^{-\alpha /d}\left\| f\right\| _{\mathbb{W}^{\alpha
,p}}.
\end{equation*}

An explicit example is the following. The torus $\mathbb{T}^{d}$ can be
partitioned into $N=n^{d}$ congruent cubes with sides $1/n$ and this
partition generates the mesh of points $\left( n^{-1}\mathbb{Z}^{d}\right)
\cap \mathbb{T}^{d}$, which gives an exact quadrature at least for all
exponentials $\exp \left( 2\pi ikx\right) $ with $\left\vert
k_{j}\right\vert <n$. In this case, (1) and (2) give an upper bound for the
error in numerical integration of the order of $N^{-\alpha /d}$. More
generally, if a manifold is decomposed into $N$ disjoint pieces $\mathcal{M}%
=U_{1}\cup U_{2}\cup ...\cup U_{N}$ with diameters $\leq cN^{-1/d}$, then
(1) gives the upper bound $N^{-\alpha /d}$. Moreover, for every $r>0$ there
are approximately $cr^{d}$ eigenfunctions with eigenvalues $\lambda
^{2}<r^{2}$ and one can choose $N\leq cr^{d}$ nodes $\left\{ z_{j}\right\}
_{j=1}^{N}$ and positive weights $\left\{ \omega _{j}\right\} _{j=1}^{N}$
which give an exact quadrature for these eigenfunctions. Then in this case
(2) gives the above upper bound $N^{-\alpha /d}$. Hence (1) and (2) imply
(3), and by (4) this latter is optimal. When the manifold is a torus or a
sphere these results are essentially known, and indeed there is a huge
literature on this subject. See \cite{Novak} for deterministic and
stochastic error bounds in numerical analysis. In particular, (3) and (4)
for $p=2$ and for spheres are contained in \cite{BH}, \cite{HS1} and
\cite{HS2}. For Besov spaces on spheres a result slightly more precise than
(3) is in \cite{SMS}, while a result slightly weaker than (4) for compact
two point homogeneous spaces is in \cite{Kushpel}. See also \cite{CS}
and, for a survey on related results, \cite{HaSa} and \cite{KS}. Beside
the proofs of (1), (2), (3), (4), which are contained in the following
section, the paper contains also a final section with a number of further
results and remarks. Among them it is proved that if a quadrature rule gives
an optimal error in the Sobolev space $\mathbb{W}^{\alpha ,2}\left( \mathcal{%
M}\right) $, then this quadrature rule is optimal also in all spaces $%
\mathbb{W}^{\beta ,2}\left( \mathcal{M}\right) $ with $d/2<\beta <\alpha $.
Moreover, it is proved that there is a relation between quadrature rules and
geometric discrepancy:

\textit{(5) If }$d\nu (x)$\textit{\ is a probability measure on }$\mathcal{M}
$\textit{, then the norm of the measure }$d\nu (x)-dx$\textit{\ as a linear
functional on }$\mathbb{W}^{\alpha ,2}\left( \mathcal{M}\right) $\textit{\
decreases as }$\alpha $\textit{\ increases. Moreover, if the norm of }$d\nu
(x)-dx$\textit{\ on }$\mathbb{W}^{\alpha ,2}\left( \mathcal{M}\right) $%
\textit{\ is }$r^{-\alpha }$\textit{,}
\begin{equation*}
\left| \int_{\mathcal{M}}f(x)d\nu (x)-\int_{\mathcal{M}}f(x)dx\right| \leq
r^{-\alpha }\left\| f\right\| _{\mathbb{W}^{\alpha ,2}},
\end{equation*}
\textit{then for every }$d/2<\beta <\alpha $\textit{\ there exists a
constant }$c$\textit{\ which may depend on }$\alpha $\textit{, }$\beta $%
\textit{, }$\mathcal{M}$\textit{, but is independent of }$r$\textit{\ and }$%
d\nu (x)$\textit{, such that}
\begin{equation*}
\left| \int_{\mathcal{M}}f(x)d\nu (x)-\int_{\mathcal{M}}f(x)dx\right| \leq
cr^{-\beta }\left\| f\right\| _{\mathbb{W}^{\beta ,2}}.
\end{equation*}

\textit{(6) Assume that for some }$r\geq 1$\textit{\ the discrepancy of the
probability measure }$d\nu (x)$\textit{\ with respect to the balls }$\left\{
B\left( y,\delta \right) \right\} $\textit{\ with center }$y$\textit{\ and
radius }$\delta $\textit{\ satisfies the estimates}
\begin{equation*}
\left| \int_{B\left( y,\delta \right) }d\nu (x)-\int_{B\left( y,\delta
\right) }dx\right| \leq \left\{
\begin{array}{l}
r^{-d}\;\;\;\text{\textit{if }}\delta \leq 1/r\text{\textit{,}} \\
r^{-1}\delta ^{d-1}\;\;\;\text{\textit{if }}\delta \geq 1/r\text{\textit{.}}%
\end{array}
\right.
\end{equation*}

\textit{Then for every }$1\leq p\leq +\infty $\textit{\ and }$\alpha >d/p$%
\textit{, there exists a constant }$c$\textit{, which may depend on }$\alpha
$\textit{\ and }$p$\textit{, but is independent of }$d\nu (x)$\textit{\ and }%
$r$\textit{, such that}
\begin{equation*}
\left| \int_{\mathcal{M}}f(x)d\nu (x)-\int_{\mathcal{M}}f(x)dx\right| \leq
\left\{
\begin{array}{l}
cr^{-\alpha }\left\| f\right\| _{\mathbb{W}^{\alpha ,p}}\;\;\;\text{\textit{%
if }}0<\alpha <1\text{\textit{,}} \\
cr^{-1}\log (1+r)\left\| f\right\| _{\mathbb{W}^{\alpha ,p}}\;\;\;\text{%
\textit{if }}\alpha =1\text{\textit{,}} \\
cr^{-1}\left\| f\right\| _{\mathbb{W}^{\alpha ,p}}\;\;\;\text{\textit{if }}%
\alpha >1\text{\textit{.}}%
\end{array}
\right.
\end{equation*}

Observe that while (1) and (2) hold for specific quadrature rules, (5) is a
result for arbitrary quadratures. Actually, (5) is only one way, from $%
\alpha $ to $\beta <\alpha $. The estimate $r^{-\alpha }$ for an $\alpha $
does not necessarily imply the estimate $cr^{-\beta }$ for $\beta >\alpha $.
Moreover, the sets $\left\{ B\left( y,\delta \right) \right\} $ in (6) are
not precisely geodesic balls, but level sets of suitable kernels on the
manifold. However, for spheres or compact rank one symmetric spaces these
sets are geodesic balls. In this case the discrepancy of the measure is the
spherical cap discrepancy. See \cite{BeCh} or \cite{Matousek}, and for
other relations between quadrature and discrepancy on spheres, see also
\cite{ABG}. Finally, we would like to point out that our paper is (almost)
self contained, it does not rely on explicit properties of manifolds or
special functions, and it may provide a unified vision and simple
alternative proofs of some known results.

\section{Main results}

The eigenfunction expansions of functions and operators are a basic tool in
what follows. The system of eigenfunctions $\left\{ \varphi _{\lambda
}(x)\right\} $ is orthonormal complete in $\mathbb{L}^{2}(\mathcal{M})$ and
to every square integrable function one can associate a Fourier transform
and series,
\begin{equation*}
\mathcal{F}f(\lambda )=\int_{\mathbb{M}}f(y)\overline{\varphi _{\lambda }(y)%
}dy,\;\;\;f(x)=\sum\limits_{\lambda }\mathcal{F}f(\lambda )\varphi
_{\lambda }(x).
\end{equation*}

Since the Laplace operator is elliptic, the eigenfunctions are smooth and it
is possible to extend the definition of Fourier transforms and series to
distributions. In particular, the Fourier expansions are always convergent,
at least in the topology of distributions. One can write the discrepancy
between integral and Riemann sum as a single integral with respect to a
measure $d\mu (x)=\sum_{j=1}^{N}\omega _{j}\delta _{z_{j}}(x)-dx$, with $%
\delta _{y}(x)$ the Dirac measure concentrated at the point $y$ and $dx$ the
Riemannian measure,
\begin{equation*}
\sum_{j=1}^{N}\omega _{j}f\left( z_{j}\right) -\int_{\mathcal{M}%
}f(x)dx=\int_{\mathcal{M}}f(x)d\mu (x).
\end{equation*}

Then the estimate of the error in the numerical integration reduces to the
estimate of the norm of a linear functional $d\mu (x)$ on a space of test
functions $f(x)$. Some of the results which follow will be stated for
generic finite signed measures $d\mu (x)$, for measures of the form $d\mu
(x)=d\nu (x)-dx$ with $d\nu (x)$ a probability measure, and also for atomic
probability measures $d\nu (x)=\sum_{j=1}^{N}\omega _{j}\delta _{z_{j}}(x)$%
. The following is an easy and straightforward extension to compact
manifolds and $p$ norms of some abstract results for reproducing kernel
Hilbert spaces. See e.g. \cite{AZ}, \cite{CF}, \cite{DLRS}.

\begin{thm}
\label{T1}\textit{\ }Let $\left\{ \psi (\lambda )\right\} $\ be a numeric
sequence indexed by the eigenvalues $\left\{ \lambda ^{2}\right\} $, with $%
\left\{ \psi (\lambda )\right\} $\ and $\left\{ \psi (\lambda )^{-1}\right\}
$\ slowly increasing, that is $\left\vert \psi (\lambda )\right\vert \leq
a\left( 1+\lambda ^{2}\right) ^{\alpha /2}$\ and $\left\vert \psi (\lambda
)^{-1}\right\vert \leq b\left( 1+\lambda ^{2}\right) ^{\beta /2}$. Let $%
A(x,y)$\ and $B(x,y)$\ be distribution kernels with Fourier transforms $%
\left\{ \psi (\lambda )\right\} $\ and $\left\{ \psi (\lambda )^{-1}\right\}
$,
\begin{equation*}
A(x,y)=\sum_{\lambda }\psi (\lambda )\varphi _{\lambda }(x)\overline{%
\varphi _{\lambda }(y)},\;\;\;B(x,y)=\sum_{\lambda }\psi (\lambda
)^{-1}\varphi _{\lambda }(x)\overline{\varphi _{\lambda }(y)}.
\end{equation*}%
Finally, let $f(x)$\ be a continuous function and let $d\mu (x)$\ be a
finite measure on $\mathcal{M}$. Then, if $1\leq p,q\leq +\infty $\ and $1/p+1/q=1$,
\begin{gather*}
\left\vert \int_{\mathcal{M}}f(x)d\mu (x)\right\vert \\
\leq \left\{ \int_{\mathcal{M}}\left\vert \int_{\mathcal{M}%
}A(x,y)f(y)dy\right\vert ^{p}dx\right\} ^{1/p}\left\{ \int_{\mathcal{M}%
}\left\vert \int_{\mathcal{M}}B(x,y)d\mu (x)\right\vert ^{q}dy\right\}
^{1/q}.
\end{gather*}%
In particular, when $p=q=2$\ and $B(x,y)=B\left( y,x\right) $\ and
\begin{equation*}
f(x)=\int_{\mathcal{M}}\int_{\mathcal{M}}B\left( x,y\right) \overline{%
B\left( y,z\right) }dy\overline{d\mu (z)},
\end{equation*}%
\ then the above inequality reduces to an equality.
\end{thm}

\begin{proof}
The assumptions $\left\{ \psi (\lambda )\right\} $\ and $\left\{ \psi
(\lambda )^{-1}\right\} $\ slowly increasing simply imply that the kernels $%
A(x,y)$ and $B(x,y)$ are tempered distributions. In what follows the pairing
between a test function and a distribution is denoted with an integral, even
when the distribution is not a function and the integral is divergent. Let
\begin{gather*}
\int_{\mathcal{M}}A(x,y)f(y)dy=\sum\limits_{\lambda }\psi (\lambda )%
\mathcal{F}f(\lambda )\varphi _{\lambda }(x), \\
\int_{\mathcal{M}}B(x,y)d\mu (y)=\sum\limits_{\lambda }\psi (\lambda )^{-1}%
\mathcal{F}\mu (\lambda )\varphi _{\lambda }(x).
\end{gather*}%
These operators are one the inverse of the other,
\begin{gather*}
f(x)=\int_{\mathcal{M}}B(x,y)\left( \int_{\mathcal{M}}A(y,z)f(z)dz\right)
dy \\
=\int_{\mathcal{M}}A(x,y)\left( \int_{\mathcal{M}}B(y,z)f(z)dz\right) dy.
\end{gather*}%
In particular, by H\"{o}lder inequality with $1/p+1/q=1$,
\begin{gather*}
\left\vert \int_{\mathcal{M}}f(x)d\mu (x)\right\vert =\left\vert \int_{%
\mathcal{M}}\int_{\mathcal{M}}\int_{\mathcal{M}}B(x,y)A(y,z)f(z)d\mu
(x)dydz\right\vert \\
\leq \left\{ \int_{\mathcal{M}}\left\vert \int_{\mathcal{M}%
}A(y,z)f(z)dz\right\vert ^{p}dy\right\} ^{1/p}\left\{ \int_{\mathcal{M}%
}\left\vert \int_{\mathcal{M}}B(x,y)d\mu (x)\right\vert ^{q}dy\right\}
^{1/q}.
\end{gather*}%
Finally, when $p=q=2$ the Cauchy inequality reduces to an equality if the
functions are proportional. Indeed, if $B\left( x,y\right) =B\left(
y,x\right) $ and
\begin{equation*}
f(x)=\int_{\mathcal{M}}\int_{\mathcal{M}}B\left( x,y\right) \overline{%
B\left( y,z\right) }dy\overline{d\mu (z)},
\end{equation*}%
then one easily verifies that
\begin{gather*}
\int_{\mathcal{M}}f(x)d\mu (x)=\int_{\mathcal{M}}\left\vert \int_{%
\mathcal{M}}B\left( x,y\right) d\mu (x)\right\vert ^{2}dy, \\
\left\{ \int_{\mathcal{M}}\left\vert \int_{\mathcal{M}}A(x,y)f(y)dy\right%
\vert ^{2}dx\right\} ^{1/2}=\left\{ \int_{\mathcal{M}}\left\vert \int_{%
\mathcal{M}}B\left( x,y\right) d\mu (x)\right\vert ^{2}dy\right\} ^{1/2}.
\end{gather*}%
Hence, when $p=q=2$ for this function the inequality in the theorem reduces
to an equality.\
\end{proof}

In what follows the operators with kernels $A(x,y)$ and $B(x,y)$ will be
powers of the Laplace Beltrami operator $\left( I+\Delta \right) ^{\pm
\alpha /2}$.

\begin{defin}
The Sobolev space $W^{\alpha ,p}\left( \mathcal{M}\right) $, $-\infty
<\alpha <+\infty $\ and $1\leq p\leq +\infty $, consists of all
distributions on $\mathcal{M}$\ with $\left( I+\Delta \right) ^{\alpha /2}f(x)$\ in $%
L^{p}\left( \mathcal{M}\right) $, that is with
\begin{equation*}
\left\Vert f\right\Vert _{\mathbb{W}^{\alpha ,p}}=\left\{ \int_{\mathcal{M}%
}\left\vert \sum_{\lambda }\left( 1+\lambda ^{2}\right) ^{\alpha /2}%
\mathcal{F}f\left( \lambda \right) \varphi _{\lambda }(x)\right\vert
^{p}dx\right\} ^{1/p}<+\infty .
\end{equation*}
\end{defin}

An equivalent definition is the following.

\begin{defin}
Let $B^{\alpha }(x,y)$, $-\infty <\alpha <+\infty $,\ be the Bessel kernel
\begin{equation*}
B^{\alpha }(x,y)=\sum_{\lambda }\left( 1+\lambda ^{2}\right) ^{-\alpha
/2}\varphi _{\lambda }(x)\overline{\varphi _{\lambda }(y)}.
\end{equation*}%
A distribution $f(x)$\ is in the Sobolev space $W^{\alpha ,p}\left( \mathcal{%
M}\right) $\ if and only if it is a Bessel potential of a function $g(x)$\
in $L^{p}\left( \mathcal{M}\right) $,
\begin{equation*}
f(x)=\int_{\mathcal{M}}B^{\alpha }(x,y)g(y)dy.
\end{equation*}%
Moreover, $\left\Vert f\right\Vert _{\mathbb{W}^{\alpha ,p}}=\left\Vert
g\right\Vert _{\mathbb{L}^{p}}$.
\end{defin}

In particular, when $p=2$,
\begin{equation*}
\left\Vert f\right\Vert _{\mathbb{W}^{\alpha ,2}}=\left\{ \sum_{\lambda
}\left( 1+\lambda ^{2}\right) ^{\alpha }\left\vert \mathcal{F}f\left(
\lambda \right) \right\vert ^{2}\right\} ^{1/2}.
\end{equation*}

Another equivalent definition is a localization result: A distribution $f(x)$
is in $\mathbb{W}^{\alpha ,p}\left( \mathcal{M}\right) $ if and only if for
every smooth function $g(x)$ with support in a local card $x=\psi (y):%
\mathbb{R}^{d}\leadsto \mathcal{M}$, the distribution $g(\psi (y))f(\psi
(y)) $ is in $\mathbb{W}^{\alpha ,p}\left( \mathbb{R}^{d}\right) $. In
particular, if $\alpha $ is a positive even integer, then $f(x)$ is in $%
\mathbb{W}^{\alpha ,p}\left( \mathcal{M}\right) $ if and only if the $p$ th
power of $f(x)$ and of $\Delta ^{\alpha /2}f(x)$ are integrable. Moreover,
distributions in $\mathbb{W}^{\alpha ,p}\left( \mathcal{M}\right) $ with $%
\alpha >d/p$ are H\"{o}lder continuous of degree $\alpha -d/p$. When applied
to functions in Sobolev classes, Theorem \ref{T1} gives the following
corollary.

\begin{cor}
\label{C2}(1) If $B^{\alpha }(x,y)=\sum_{\lambda }\left( 1+\lambda
^{2}\right) ^{-\alpha /2}\varphi _{\lambda }(x)\overline{\varphi _{\lambda
}(y)}$\ is the Bessel kernel, if $d\mu (x)$\ is a finite measure on $\mathcal{M}$, and
if $f(x)$\ is a continuous function in $W^{\alpha ,p}\left( \mathcal{M}%
\right) $, with $1\leq p,q\leq +\infty $\ and $1/p+1/q=1$, then
\begin{equation*}
\left\vert \int_{\mathcal{M}}f(x)d\mu (x)\right\vert \leq \left\{ \int_{%
\mathcal{M}}\left\vert \int_{\mathcal{M}}B^{\alpha }(x,y)d\mu
(x)\right\vert ^{q}dy\right\} ^{1/q}\left\Vert f\right\Vert _{\mathbb{W}%
^{\alpha ,p}}.
\end{equation*}%
If $\alpha >d/p$\ then the above integrals are well-defined and finite. On
the contrary, the spaces $W^{\alpha ,p}\left( \mathcal{M}\right) $\ with $%
\alpha \leq d/p$\ contain unbounded functions and, if the measure $d\mu (x)$%
\ does not vanish on the set where $f(x)=\infty $, then $\int_{\mathcal{M}%
}f(x)d\mu (x)$\ may diverge.

(2) When $p=q=2$\ then the above inequality simplifies,
\begin{equation*}
\left\vert \int_{\mathcal{M}}f(x)d\mu (x)\right\vert \leq \left\{ \int_{%
\mathcal{M}}\int_{\mathcal{M}}B^{2\alpha }\left( x,y\right) d\mu (x)%
\overline{d\mu (y)}\right\} ^{1/2}\left\Vert f\right\Vert _{\mathbb{W}%
^{\alpha ,2}}
\end{equation*}%
Equivalently, by the Fourier expansion of the Bessel kernel,
\begin{equation*}
\left\vert \int_{\mathcal{M}}f(x)d\mu (x)\right\vert \leq \left\{
\sum_{\lambda }\left( 1+\lambda ^{2}\right) ^{-\alpha }\left\vert \mathcal{F%
}\mu \left( \lambda \right) \right\vert ^{2}\right\} ^{1/2}\left\Vert
f\right\Vert _{\mathbb{W}^{\alpha ,2}}.
\end{equation*}%
Moreover, with $f(x)=\int_{\mathcal{M}}B^{2\alpha }\left( x,y\right)
\overline{d\mu (y)}$\ the above inequalities reduce to equalities.

(3) If $d\mu (x)=d\nu (x)-dx$\ is the difference between a probability
measure $d\nu (x)$\ and the Riemannian measure $dx$, then
\begin{equation*}
\left\vert \int_{\mathcal{M}}f(x)d\nu (x)-\int_{\mathcal{M}%
}f(x)dx\right\vert \leq \left\{ \int_{\mathcal{M}}\int_{\mathcal{M}%
}B^{2\alpha }\left( x,y\right) d\nu (x)d\nu (y)-1\right\} ^{1/2}\left\Vert
f\right\Vert _{\mathbb{W}^{\alpha ,2}}.
\end{equation*}%
Equivalently,
\begin{equation*}
\left\vert \int_{\mathcal{M}}f(x)d\nu (x)-\int_{\mathcal{M}%
}f(x)dx\right\vert \leq \left\{ \sum_{\lambda >0}\left( 1+\lambda
^{2}\right) ^{-\alpha }\left\vert \mathcal{F}\nu \left( \lambda \right)
\right\vert ^{2}\right\} ^{1/2}\left\Vert f\right\Vert _{\mathbb{W}^{\alpha
,2}}.
\end{equation*}
\end{cor}

\begin{proof}
(1) is an immediate corollary of Theorem \ref{T1}. In order
to prove (2), observe that
\begin{equation*}
\int_{\mathcal{M}}B^{\alpha }(x,y)B^{\beta }(y,z)dy=B^{\alpha +\beta
}\left( x,z\right) .
\end{equation*}%
Moreover, this Bessel kernel is real and symmetric. Hence,
\begin{gather*}
\int_{\mathcal{M}}\left\vert \int_{\mathcal{M}}B^{\alpha }(x,y)d\mu
(x)\right\vert ^{2}dy \\
=\int_{\mathcal{M}}\int_{\mathcal{M}}\int_{\mathcal{M}}B^{\alpha
}(x,y)B^{\alpha }(z,y)dyd\mu (x)\overline{d\mu (z)} \\
=\int_{\mathcal{M}}\int_{\mathcal{M}}B^{2\alpha }(x,z)d\mu (x)\overline{%
d\mu (z)}.
\end{gather*}

(3) is a corollary of (1) and (2). Indeed, since $B^{2\alpha }\left(
x,y\right) =B^{2\alpha }\left( y,x\right) $ and $\int_{\mathcal{M}%
}B^{2\alpha }\left( x,y\right) dy=1$, it follows that
\begin{gather*}
\int_{\mathcal{M}}\int_{\mathcal{M}}B^{2\alpha }\left( x,y\right) \left(
d\nu (x)-dx\right) \left( d\nu (y)-dy\right) \\
=\int_{\mathcal{M}}\int_{\mathcal{M}}B^{2\alpha }\left( x,y\right) d\nu
(x)d\nu (y)-\int_{\mathcal{M}}\int_{\mathcal{M}}B^{2\alpha }\left(
x,y\right) d\nu (x)dy \\
-\int_{\mathcal{M}}\int_{\mathcal{M}}B^{2\alpha }\left( x,y\right) dxd\nu
(y)+\int_{\mathcal{M}}\int_{\mathcal{M}}B^{2\alpha }\left( x,y\right) dxdy
\\
=\int_{\mathcal{M}}\int_{\mathcal{M}}B^{2\alpha }\left( x,y\right) d\nu
(x)d\nu (y)-1.
\end{gather*}%
Finally, by Sobolev imbedding theorem, functions in $\mathbb{W}^{\alpha
,p}\left( \mathcal{M}\right) $\ with $\alpha >d/p$ are continuous and all
the above integrals are well-defined and finite. This also follows from
Lemma \ref{L4} and Remark \ref{R15} below.\
\end{proof}

The above corollary leads to estimate the energy integrals
\begin{gather*}
\left\{ \int_{\mathcal{M}}\left| \int_{\mathcal{M}}B^{\alpha }(x,y)d\mu
(x)\right| ^{q}dy\right\} ^{1/q}, \\
\left\{ \int_{\mathcal{M}}\int_{\mathcal{M}}B^{2\alpha }\left( x,y\right)
d\nu (x)d\nu (y)-1\right\} ^{1/2}=\left\{ \sum_{\lambda >0}\left( 1+\lambda
^{2}\right) ^{-\alpha }\left| \mathcal{F}\nu \left( \lambda \right) \right|
^{2}\right\} ^{1/2}.
\end{gather*}

By the last formula, the energy attains a minimum if and only if $\mathcal{F}%
\nu \left( \lambda \right) =0$ for all $\lambda >0$, and this gives the
Riemannian measure $dx$. The meaning of the corollary is that measures with
low energy are close to the Riemannian measure and they give good quadrature
rules. In order to give quantitative estimates for the above integrals, one
has to collect some properties of the Bessel kernels. The norm of the
function $y\leadsto B^{\alpha }(x,y)$ in $\mathbb{W}^{\gamma ,2}\left(
\mathcal{M}\right) $ is
\begin{equation*}
\left\| B^{\alpha }(x,\cdot )\right\| _{\mathbb{W}^{\gamma ,2}}=\left\{
\sum_{\lambda }\left( 1+\lambda ^{2}\right) ^{\gamma -\alpha }\left|
\varphi _{\lambda }(x)\right| ^{2}\right\} ^{1/2}.
\end{equation*}

By Weyl's estimates on the spectrum of an elliptic operator, see Theorem
17.5.3 in \cite{Hormander}, for every $r>1$ there are approximately $cr^{d}$
eigenfunctions $\varphi _{\lambda }(x)$ with eigenvalues $\lambda ^{2}<r^{2}$
and $\sum_{\lambda \leq r}\left\vert \varphi _{\lambda }(x)\right\vert
^{2}\leq cr^{d}$. It then follows that the norm in $\mathbb{W}^{\gamma
,2}\left( \mathcal{M}\right) $ of $B^{\alpha }(x,y)$ is finite provided that
$\gamma <\alpha -d/2$ and, by Sobolev imbedding theorem, it also follows
that $B^{\alpha }(x,y)$ is H\"{o}lder continuous of degree $\delta <\alpha -d
$. Indeed, we shall see that a bit more is true: $B^{\alpha }(x,y)$ is H\"{o}%
lder continuous of degree $\alpha -d$.

\begin{lem}
\textbf{\label{L3}}\textit{\ }The heat kernel $W\left( t,x,y\right)
=\sum_{\lambda }\exp \left( -\lambda ^{2}t\right) \varphi _{\lambda }(x)%
\overline{\varphi _{\lambda }(y)}$, which is the fundamental solution to the
heat equation $\partial /\partial t=-\Delta $\ on $\mathbb{R}_{+}\times \mathcal{M}$, is
symmetric real and positive, $W\left( t,x,y\right) =W\left( t,y,x\right) >0$%
\ for every $x,y\in \mathcal{M}$\ and $t>0$. Moreover, for every $m$\ and $n$\ there
exists $c$\ such that, if $\left\vert x-y\right\vert $\ denotes the
Riemannian distance between $x$\ and $y$, and $\nabla $\ the gradient,
\begin{equation*}
\left\{
\begin{array}{l}
\left\vert \nabla ^{m}W\left( t,x,y\right) \right\vert \leq
ct^{-(d+m)/2}\left( 1+\left\vert x-y\right\vert /\sqrt{t}\right) ^{-n}%
\mathit{\;\;\;}\text{\textit{if }}0<t\leq 1\text{\textit{,}} \\
\left\vert \nabla ^{m}W\left( t,x,y\right) \right\vert \leq c\mathit{\;\;\;}%
\text{\textit{if }}1\leq t<+\infty \text{\textit{.}}%
\end{array}%
\right.
\end{equation*}
\end{lem}

\begin{proof}
All of this is well known. The idea is that heat has essentially a finite
speed of propagation and diffusion in manifolds is comparable to diffusion
in Euclidean spaces, at least for small times. The heat kernel in the
Euclidean space $\mathbb{R}^{d}$ is a Gaussian,
\begin{gather*}
W\left( t,x,y\right) =\int_{\mathbb{R}^{d}}\exp \left( -4\pi
^{2}t\left\vert \xi \right\vert ^{2}\right) \exp \left( 2\pi i\left(
x-y\right) \xi \right) d\xi  \\
=\left( 4\pi t\right) ^{-d/2}\exp \left( -\left\vert x-y\right\vert
^{2}/4t\right) .
\end{gather*}%
By the Poisson summation formula, the heat kernel on the torus $\mathbb{T}%
^{d}=\mathbb{R}^{d}/\mathbb{Z}^{d}$ is the periodized of the kernel in the
space,
\begin{gather*}
\sum_{k\in \mathbb{Z}^{d}}\exp \left( -4\pi ^{2}\left\vert k\right\vert
^{2}t\right) \exp \left( 2\pi ik\left( x-y\right) \right)  \\
=\sum_{k\in \mathbb{Z}^{d}}\left( 4\pi t\right) ^{-d/2}\exp \left(
-\left\vert x-y-k\right\vert ^{2}/4t\right) .
\end{gather*}%
When $x$ is close to $y$ and $t$ is small, the main contribution to the sum
comes from the term with $k=0$,
\begin{equation*}
W\left( t,x,y\right) \approx \left( 4\pi t\right) ^{-d/2}\exp \left(
-\left\vert x-y\right\vert ^{2}/4t\right) .
\end{equation*}%
The remainder gives a bounded contribution,
\begin{equation*}
\left\vert \sum_{k\in \mathbb{Z}^{d}-\left\{ 0\right\} }\left( 4\pi
t\right) ^{-d/2}\exp \left( -\left\vert x-y-k\right\vert ^{2}/4t\right)
\right\vert \leq c.
\end{equation*}%
Analogous estimates hold for the derivatives. This proves the lemma for the
torus. The heat kernel on a compact manifold is similar, in particular it
has an asymptotic expansion with euclidean main term. See e.g. \cite{Chavel}%
, Chapter VI. More precisely, by the Minakshisundaram Pleijel recursion
formulas, there exist smooth functions $\left\{ u_{k}\left( x,y\right)
\right\} $ such that, if $t$ is small and $\left\vert x-y\right\vert $
denotes the distance between $x$ and $y$,
\begin{equation*}
W\left( t,x,y\right) \approx (4\pi t)^{-d/2}\exp \left( -\left\vert
x-y\right\vert ^{2}/4t\right) \sum_{k=0}^{n}t^{k}u_{k}\left( x,y\right)
+O\left( t^{n+1}\right) .
\end{equation*}%
On the contrary, $W\left( t,x,y\right) \rightarrow 1$ when $t\rightarrow
+\infty $. The estimates on the size of this kernel and its derivatives are
a consequence of this asymptotic expansion. The positivity $W\left(
t,x,y\right) >0$ is a consequence of the maximum principle for heat equation
and the symmetry $W\left( t,x,y\right) =W\left( t,y,x\right) $ follows from
this positivity and the eigenfunction expansion.\
\end{proof}

\begin{lem}
\textbf{\label{L4}} (1) The Bessel kernel $B^{\alpha }(x,y)$\ with $\alpha
>0 $\ is a superposition of heat kernels $W\left( t,x,y\right) $:
\begin{equation*}
B^{\alpha }(x,y)=\Gamma \left( \alpha /2\right) ^{-1}\int_{0}^{+\infty
}t^{\alpha /2-1}\exp \left( -t\right) W\left( t,x,y\right) dt.
\end{equation*}

(2) The Bessel kernel $B^{\alpha }(x,y)$\ with $\alpha >0$\ is real and
positive for every $x,y\in \mathcal{M}$, and it is smooth in $\left\{ x\neq y\right\} $%
. Moreover, for suitable constants $0<a<b$,
\begin{gather*}
a\left\vert x-y\right\vert ^{\alpha -d}\leq B^{\alpha }(x,y)\leq b\left\vert
x-y\right\vert ^{\alpha -d}\;\;\;\text{\textit{if }}0<\alpha <d\text{\textit{%
,}} \\
a\log \left( 1+\left\vert x-y\right\vert ^{-1}\right) \leq B^{\alpha
}(x,y)\leq b\log \left( 1+\left\vert x-y\right\vert ^{-1}\right) \;\;\;\text{%
\textit{if }}\alpha =d\text{\textit{,}} \\
a\leq B^{\alpha }(x,y)\leq b\;\;\;\text{\textit{if }}\alpha >d\text{\textit{.%
}}
\end{gather*}

(3) If $d<\alpha <d+1$, then $B^{\alpha }(x,y)$\ is H\"{o}lder continuous of
degree $\alpha -d$, that is there exists $c$\ such that for every $x,y,z\in
\mathcal{M} $,
\begin{equation*}
\left\vert B^{\alpha }(x,y)-B^{\alpha }(x,z)\right\vert \leq c\left\vert
y-z\right\vert ^{\alpha -d}.
\end{equation*}

(4) If $d<\alpha <d+2$, then there exists $c$\ such that for every $x,y\in \mathcal{M}$%
,
\begin{equation*}
\left\vert B^{\alpha }(x,x)-B^{\alpha }(x,y)\right\vert \leq c\left\vert
x-y\right\vert ^{\alpha -d}.
\end{equation*}
\end{lem}

\begin{proof}
When the manifold is a torus and the eigenfunctions are exponentials the
proof is elementary. The Bessel kernel in the torus $\mathbb{T}^{d}$ is an
even function and it is sum of cosines,
\begin{gather*}
B^{\alpha }(x,y)=\sum_{k\in \mathbb{Z}^{d}}\left( 1+4\pi ^{2}\left\vert
k\right\vert ^{2}\right) ^{-\alpha /2}\exp \left( 2\pi ikx\right) \exp
\left( -2\pi iky\right) \\
=\sum_{k\in \mathbb{Z}^{d}}\left( 1+4\pi ^{2}\left\vert k\right\vert
^{2}\right) ^{-\alpha /2}\cos \left( 2\pi k\left( x-y\right) \right) .
\end{gather*}%
Hence,
\begin{gather*}
B^{\alpha }(x,x)-B^{\alpha }(x,y)=2\sum_{k\in \mathbb{Z}^{d}}\left( 1+4\pi
^{2}\left\vert k\right\vert ^{2}\right) ^{-\alpha /2}\sin ^{2}\left( \pi
k\left( x-y\right) \right) \\
\leq 2\pi ^{2}\left\vert x-y\right\vert ^{2}\sum_{\left\vert k\right\vert
\leq \left\vert x-y\right\vert ^{-1}}\left\vert k\right\vert ^{2}\left(
1+4\pi ^{2}\left\vert k\right\vert ^{2}\right) ^{-\alpha
/2}+2\sum_{\left\vert k\right\vert >\left\vert x-y\right\vert ^{-1}}\left(
1+4\pi ^{2}\left\vert k\right\vert ^{2}\right) ^{-\alpha /2} \\
\leq \left\{
\begin{array}{l}
c\left\vert x-y\right\vert ^{\alpha -d}\;\;\;\text{if }d<\alpha <d+2\text{,}
\\
c\left\vert x-y\right\vert ^{2}\log \left( 1+\left\vert x-y\right\vert
^{-1}\right) \;\;\;\text{if }\alpha =d+2\text{,} \\
c\left\vert x-y\right\vert ^{2}\;\;\;\text{if }\alpha >d+2\text{.}%
\end{array}%
\right.
\end{gather*}%
Also observe that the series which defines $B^{\alpha }(x,x)-B^{\alpha
}(x,y) $ has positive terms and the above inequalities can be reversed. This
proves (4) for a torus, and the proof of (3) and (2) is similar. A proof for
a generic manifold follows from the representation of Bessel kernels as
superposition of heat kernels and the estimates in the previous lemma. In
particular, (1) follows from the identity for the Gamma function
\begin{equation*}
\left( 1+\lambda ^{2}\right) ^{-\alpha /2}=\Gamma \left( \alpha /2\right)
^{-1}\int_{0}^{+\infty }t^{\alpha /2-1}\exp \left( -t\left( 1+\lambda
^{2}\right) \right) dt.
\end{equation*}%
By Lemma \ref{L3}, for every $n$,
\begin{equation*}
0<W\left( t,x,y\right) \leq \left\{
\begin{array}{l}
ct^{\left( n-d\right) /2}\left\vert x-y\right\vert ^{-n}\;\;\;\text{if }%
0<t\leq \left\vert x-y\right\vert ^{2}\text{,} \\
ct^{-d/2}\;\;\;\text{if }\left\vert x-y\right\vert ^{2}\leq t\leq 1\text{,}
\\
c\;\;\;\text{if }t\geq 1\text{.}%
\end{array}%
\right.
\end{equation*}%
Hence, if $0<\alpha <d$ and $n>d-\alpha $,
\begin{gather*}
B^{\alpha }(x,y)=\Gamma \left( \alpha /2\right) ^{-1}\int_{0}^{+\infty
}t^{\alpha /2-1}\exp \left( -t\right) W\left( t,x,y\right) dt \\
\leq c\left\vert x-y\right\vert ^{-n}\int_{0}^{\left\vert x-y\right\vert
^{2}}t^{\left( \alpha +n-d\right) /2-1}dt+c\int_{\left\vert x-y\right\vert
^{2}}^{1}t^{\left( \alpha -d\right) /2-1}dt+\int_{1}^{+\infty }t^{\alpha
/2-1}\exp \left( -t\right) dt \\
\leq c\left\vert x-y\right\vert ^{\alpha -d}.
\end{gather*}%
Indeed one can easily see that these inequalities can be reversed. Hence $%
B^{\alpha }(x,y)\approx c\left\vert x-y\right\vert ^{\alpha -d}$. This
proves (2) when $0<\alpha <d$, and the proofs of the cases $\alpha =d$ and $%
\alpha >d$ are similar. Also the proof of (3) is similar. Write
\begin{gather*}
B^{\alpha }(x,y)-B^{\alpha }(x,z) \\
=\Gamma \left( \alpha /2\right) ^{-1}\int_{0}^{+\infty }t^{\alpha /2-1}\exp
\left( -t\right) \left( W\left( t,x,y\right) -W\left( t,x,z\right) \right)
dt.
\end{gather*}%
Then recall that, by Lemma \ref{L3},
\begin{equation*}
\left\vert W\left( t,x,y\right) -W\left( t,x,z\right) \right\vert \leq
\left\{
\begin{array}{l}
ct^{-d/2}\;\;\;\text{if }0<t\leq \left\vert y-z\right\vert ^{2}\text{,} \\
ct^{-(d+1)/2}\left\vert y-z\right\vert \;\;\;\text{if }\left\vert
y-z\right\vert ^{2}\leq t\leq 1\text{,} \\
c\left\vert y-z\right\vert \;\;\;\text{if }t\geq 1\text{.}%
\end{array}%
\right.
\end{equation*}%
Hence,
\begin{gather*}
\left\vert B^{\alpha }(x,y)-B^{\alpha }(x,z)\right\vert \leq
c\int_{0}^{\left\vert y-z\right\vert ^{2}}t^{(\alpha -d)/2-1}\exp \left(
-t\right) dt \\
+c\left\vert y-z\right\vert \int_{\left\vert y-z\right\vert
^{2}}^{1}t^{(\alpha -d-1)/2-1}\exp \left( -t\right) dt+c\left\vert
y-z\right\vert \int_{1}^{+\infty }t^{\alpha /2-1}\exp \left( -t\right) dt \\
\leq c\left\vert y-z\right\vert ^{\alpha -d}.
\end{gather*}%
Finally, the estimate for $\left\vert B^{\alpha }(x,x)-B^{\alpha
}(x,y)\right\vert $ in (4) is analogous to the previous one, but it holds in
a larger range of $\alpha $. It suffices to observe that $W\left(
t,x,y\right) $ is stationary at $x=y$ and it satisfies the estimates
\begin{equation*}
\left\vert W\left( t,x,x\right) -W\left( t,x,y\right) \right\vert \leq
\left\{
\begin{array}{l}
ct^{-d/2}\;\;\;\text{if }0<t\leq \left\vert x-y\right\vert ^{2}\text{,} \\
ct^{-d/2-1}\left\vert x-y\right\vert ^{2}\;\;\;\text{if }\left\vert
x-y\right\vert ^{2}\leq t\leq 1\text{,} \\
c\left\vert x-y\right\vert ^{2}\;\;\;\text{if }t\geq 1\text{.\ }%
\end{array}%
\right.
\end{equation*}
\end{proof}

The following is Result (1) in the Introduction.

\begin{thm}
\label{T5}For every $d/2<\alpha <d/2+1$\ there exists $c>0$\ with the
following property: If $\mathcal{M}=U_{1}\cup U_{2}\cup ...\cup U_{N}$\ is a
decomposition of $\mathcal{M}$\ in disjoint pieces with measure $\left\vert
U_{j}\right\vert =\omega _{j}$, then there exists a distribution of points $%
\left\{ z_{j}\right\} _{j=1}^{N}$\ with $z_{j}\in U_{j}$\ such that
\begin{equation*}
\left\vert \sum_{j=1}^{N}\omega _{j}f\left( z_{j}\right) -\int_{\mathcal{M}%
}f(x)dx\right\vert \leq c\max_{1\leq j\leq N}\left\{ \mathrm{diameter}%
\left( U_{j}\right) ^{\alpha }\right\} \left\Vert f\right\Vert _{\mathbb{W}%
^{\alpha ,2}\left( \mathcal{M}\right) }.
\end{equation*}
\end{thm}

\begin{proof}
By Corollary \ref{C2} (3), with $d\nu (x)=\sum_{j=1}^{N}\omega _{j}\delta
_{z_{j}}(x)$,
\begin{equation*}
\left\vert \int_{\mathcal{M}}f(x)d\nu (x)-\int_{\mathcal{M}%
}f(x)dx\right\vert \leq \left\{ \sum_{i=1}^{N}\sum_{j=1}^{N}\omega
_{i}\omega _{j}B^{2\alpha }\left( z_{i},z_{j}\right) -1\right\}
^{1/2}\left\Vert f\right\Vert _{\mathbb{W}^{\alpha ,2}}.
\end{equation*}%
It suffices to compute the average value of $\sum_{i=1}^{N}\sum_{j=1}^{N}%
\omega _{i}\omega _{j}B^{2\alpha }\left( z_{i},z_{j}\right) -1$ on $%
U_{1}\times U_{2}\times ...\times U_{N}$ with respect to the probability
measures $\omega _{j}^{-1}dz_{j}$ uniformly distributed on $U_{j}$. First
observe that
\begin{gather*}
\left( \prod_{k=1}^{N}\omega _{k}^{-1}\right)
\int_{U_{1}}...\int_{U_{N}}dz_{1}...dz_{N}=1, \\
1=\int_{\mathcal{M}}\int_{\mathcal{M}}B^{2\alpha }\left( x,y\right)
dxdy=\sum_{i=1}^{N}\sum_{j=1}^{N}\int_{U_{i}}\int_{U_{j}}B^{2\alpha
}\left( x,y\right) dxdy.
\end{gather*}%
Then,
\begin{gather*}
\left( \prod_{k=1}^{N}\omega _{k}^{-1}\right)
\int_{U_{1}}...\int_{U_{N}}\left( \sum_{i=1}^{N}\sum_{j=1}^{N}\omega
_{i}\omega _{j}B^{2\alpha }\left( z_{i},z_{j}\right) -1\right)
dz_{1}...dz_{N} \\
=\sum_{j}\omega _{j}\int_{U_{j}}B^{2\alpha }\left( z_{j},z_{j}\right)
dz_{j}+\underset{i\not=j}{\sum \sum }\int_{U_{i}}\int_{U_{j}}B^{2\alpha
}\left( z_{i},z_{j}\right) dz_{i}dz_{j} \\
-\sum_{j}\int_{U_{j}}\int_{U_{j}}B^{2\alpha }\left( x,y\right) dxdy-%
\underset{i\not=j}{\sum \sum }\int_{U_{i}}\int_{U_{j}}B^{2\alpha }\left(
x,y\right) dxdy \\
=\sum_{j=1}^{N}\int_{U_{j}}\int_{U_{j}}\left( B^{2\alpha }\left(
x,x\right) -B^{2\alpha }\left( x,y\right) \right) dxdy.
\end{gather*}%
Since, by Lemma \ref{L4} (4), $\left\vert B^{2\alpha }\left( x,x\right)
-B^{2\alpha }\left( x,y\right) \right\vert \leq c\left\vert x-y\right\vert
^{2\alpha -d}$ when $d<2\alpha <d+2$, and since $\omega _{j}=\left\vert
U_{j}\right\vert \leq c\,\mathrm{diameter}\left( U_{j}\right) ^{d}$,
\begin{gather*}
\sum_{j=1}^{N}\int_{U_{j}}\int_{U_{j}}\left\vert B^{2\alpha }\left(
x,x\right) -B^{2\alpha }\left( x,y\right) \right\vert dxdy \\
\leq \sum_{j=1}^{N}\left\vert U_{j}\right\vert ^{2}\sup \left\{ \left\vert
B^{2\alpha }\left( x,x\right) -B^{2\alpha }\left( x,y\right) \right\vert
,\;x,y\in U_{j}\right\} \\
\leq c\sum_{j=1}^{N}\left\vert U_{j}\right\vert ^{2}\mathrm{diameter}%
\left( U_{j}\right) ^{2\alpha -d}\leq c\sum_{j=1}^{N}\left\vert
U_{j}\right\vert \mathrm{diameter}\left( U_{j}\right) ^{2\alpha }.\;
\end{gather*}
\end{proof}

For the next result we shall need estimates for partial sums of Fourier
expansions of the Bessel kernels.

\begin{lem}
\textbf{\label{L6}}\textit{\ }Let $\chi \left( \lambda \right) $\ be an even
smooth function on $-\infty <\lambda <+\infty $\ with support in $1/2\leq
\left\vert \lambda \right\vert \leq 2$\ and let
\begin{equation*}
P^{\alpha }(r,x,y)=\sum_{\lambda }\chi \left( \lambda /r\right) \left(
1+\lambda ^{2}\right) ^{-\alpha /2}\varphi _{\lambda }(x)\overline{\varphi
_{\lambda }(y)}.
\end{equation*}%
Then for every $n>0$\ there exists $c$\ such that for every $r>1$\ and $%
x,y\in M$,
\begin{equation*}
\left\vert P^{\alpha }(r,x,y)\right\vert \leq cr^{d-\alpha }\left(
1+r\left\vert x-y\right\vert \right) ^{-n}.
\end{equation*}
\end{lem}

\begin{proof}
The numerology behind this estimate is quite simple. The approximation of
the Bessel kernel $B^{\alpha }(x,y)$ by linear combinations of
eigenfunctions with eigenvalues $\lambda ^{2}<r^{2}$ is localized and only
points $x$ and $y$ with $\left\vert x-y\right\vert \leq 1/r$ really matter.
In particular, since $B^{\alpha }(x,y)$ is smooth away from the diagonal, at
distance $\left\vert x-y\right\vert \leq 1/r$ the approximation is rough,
but at distance $\left\vert x-y\right\vert \geq 1/r$ it is quite good. The
analogue of $P^{\alpha }(r,x,y)$ in the Euclidean setting is the kernel
\begin{gather*}
Q\left( r,x-y\right) =\int_{\mathbb{R}^{d}}\chi \left( 2\pi \left\vert \xi
\right\vert /r\right) \left( 1+4\pi ^{2}\left\vert \xi \right\vert
^{2}\right) ^{-\alpha /2}\exp \left( 2\pi i\left( x-y\right) \xi \right)
d\xi  \\
=r^{d}\int_{\mathbb{R}^{d}}\chi \left( 2\pi \left\vert \xi \right\vert
\right) \left( 1+4\pi ^{2}r^{2}\left\vert \xi \right\vert ^{2}\right)
^{-\alpha /2}\exp \left( 2\pi ir\left( x-y\right) \xi \right) d\xi .
\end{gather*}%
Since $\chi \left( 2\pi \left\vert \xi \right\vert \right) $ has support in $%
1/2\leq 2\pi \left\vert \xi \right\vert \leq 2$, for every $r$ and $x,y\in
\mathbb{R}^{d}$ one has
\begin{gather*}
\left\vert r^{d}\int_{\mathbb{R}^{d}}\chi \left( 2\pi \left\vert \xi
\right\vert \right) \left( 1+4\pi ^{2}r^{2}\left\vert \xi \right\vert
^{2}\right) ^{-\alpha /2}\exp \left( 2\pi ir\left( x-y\right) \xi \right)
d\xi \right\vert  \\
\leq r^{d-\alpha }\int_{\mathbb{R}^{d}}\left( 2\pi \left\vert \xi
\right\vert \right) ^{-\alpha }\left\vert \chi \left( 2\pi \left\vert \xi
\right\vert \right) \right\vert d\xi \leq cr^{d-\alpha }.
\end{gather*}%
This estimate can be improved in the range $\left\vert x-y\right\vert \geq
1/r$. Indeed, an integration by parts gives
\begin{gather*}
r^{d}\int_{\mathbb{R}^{d}}\chi \left( 2\pi \left\vert \xi \right\vert
\right) \left( 1+4\pi ^{2}r^{2}\left\vert \xi \right\vert ^{2}\right)
^{-\alpha /2}\exp \left( 2\pi ir\left( x-y\right) \xi \right) d\xi  \\
=r^{d}\int_{\mathbb{R}^{d}}\chi \left( 2\pi \left\vert \xi \right\vert
\right) \left( 1+4\pi ^{2}r^{2}\left\vert \xi \right\vert ^{2}\right)
^{-\alpha /2}\Delta _{\xi }^{n}\left( \left( 4\pi ^{2}r^{2}\left\vert
x-y\right\vert ^{2}\right) ^{-n}\exp \left( 2\pi ir\left( x-y\right) \xi
\right) \right) d\xi  \\
=r^{d}\left( 4\pi ^{2}r^{2}\left\vert x-y\right\vert ^{2}\right) ^{-n}\int_{%
\mathbb{R}^{d}}\exp \left( 2\pi ir\left( x-y\right) \xi \right) \Delta _{\xi
}^{n}\left( \chi \left( 2\pi \left\vert \xi \right\vert \right) \left(
1+4\pi ^{2}r^{2}\left\vert \xi \right\vert ^{2}\right) ^{-\alpha /2}\right)
d\xi .
\end{gather*}%
Hence,
\begin{gather*}
\left\vert r^{d}\int_{\mathbb{R}^{d}}\chi \left( 2\pi \left\vert \xi
\right\vert \right) \left( 1+4\pi ^{2}r^{2}\left\vert \xi \right\vert
^{2}\right) ^{-\alpha /2}\exp \left( 2\pi ir\left( x-y\right) \xi \right)
d\xi \right\vert  \\
\leq r^{d}\left( 4\pi ^{2}r^{2}\left\vert x-y\right\vert ^{2}\right)
^{-n}\int_{\mathbb{R}^{d}}\left\vert \Delta _{\xi }^{n}\left( \chi \left(
2\pi \left\vert \xi \right\vert \right) \left( 1+4\pi ^{2}r^{2}\left\vert
\xi \right\vert ^{2}\right) ^{-\alpha /2}\right) \right\vert d\xi  \\
\leq cr^{d-\alpha -2n}\left\vert x-y\right\vert ^{-2n}.
\end{gather*}%
Now it suffices to transfer these estimates from the Euclidean space to the
manifold. For the torus, this can be done via the Poisson summation formula.
If $Q\left( r,x-y\right) $ is the truncated Bessel kernel in $\mathbb{R}^{d}$
defined above, then the truncated Bessel kernel in $\mathbb{T}^{d}$ is
\begin{equation*}
\sum_{k\in \mathbb{Z}^{d}}\chi \left( 2\pi \left\vert k\right\vert
/r\right) \left( 1+4\pi ^{2}\left\vert h\right\vert ^{2}\right) ^{-\alpha
/2}\exp \left( 2\pi ik\left( x-y\right) \right) =\sum_{k\in \mathbb{Z}%
^{d}}Q\left( r,x-y+k\right) .
\end{equation*}%
When $\left\vert x_{j}-y_{j}\right\vert \leq 1/2$, the main term in the last
sum is the one with $k=0$, while the contribution of terms with $k\neq 0$ is
negligible,
\begin{gather*}
\left\vert Q\left( r,x-y\right) \right\vert \leq cr^{d-\alpha }\left(
1+r\left\vert x-y\right\vert \right) ^{-n}, \\
\sum_{k\in \mathbb{Z}^{d}-\left\{ 0\right\} }\left\vert Q\left(
r,x-y-k\right) \right\vert \leq cr^{d-\alpha -n}.
\end{gather*}%
Finally, the estimate for the truncated Bessel kernel on a generic manifold
can be obtained by transference from $\mathbb{R}^{d}$ via pseudodifferential
techniques. For more details, see e.g. \cite{Taylor} Chapter XII, or \cite%
{BrCo}.\
\end{proof}

The following is a result on the homogeneity of measures which appear in
quadrature rules and it gives sharp estimates of the discrepancy of such
measures. Similar estimates on spheres are in \cite{ABG}.

\begin{lem}
\label{L7}\ Assume that $d\nu (x)$\ is a probability measure on $\mathcal{M}$\ with
the property that for every eigenfunction $\varphi _{\lambda }(x)$\ with
eigenvalues $\lambda ^{2}<r^{2}$,
\begin{equation*}
\int_{\mathcal{M}}\varphi _{\lambda }(x)d\nu (x)=\int_{\mathcal{M}}\varphi
_{\lambda }(x)dx.
\end{equation*}%
Then for every $n$\ there exists $c$, which may depend on $n$\ and $\mathcal{M}$, but
is independent of $r$\ and $d\nu (x)$, such that for every measurable set $%
\Omega $\ in $\mathcal{M}$,
\begin{equation*}
\left\vert \int_{\Omega }d\nu (x)-\int_{\Omega }dx\right\vert \leq c\int_{%
\mathcal{M}}\left( 1+r\mathrm{distance}\left\{ x,\partial \Omega \right\}
\right) ^{-n}dx.
\end{equation*}%
In particular, the discrepancy between the measures $d\nu (x)$\ and $dx$\
with respect to balls $\left\{ \left\vert x-y\right\vert \leq s\right\} $\
is dominated by
\begin{equation*}
\left\vert \int_{\left\{ \left\vert x-y\right\vert \leq s\right\} }d\nu
(x)-\int_{\left\{ \left\vert x-y\right\vert \leq s\right\} }dx\right\vert
\leq \left\{
\begin{array}{l}
cr^{-d}\;\;\;\text{\textit{if }}s\leq 1/r\text{\textit{,}} \\
cr^{-1}s^{d-1}\;\;\;\text{\textit{if }}s\geq 1/r\text{\textit{.}}%
\end{array}%
\right.
\end{equation*}
\end{lem}

\begin{proof}
It is proved in \cite{CGT} that given $n$, there exists $c$ such that for
every measurable set $\Omega $ in $\mathcal{M}$ and every $r>0$ there exist
two linear combinations of eigenfunctions $A(x)=\sum_{\lambda <r}a\left(
\lambda \right) \varphi _{\lambda }(x)$ and $B(x)=\sum_{\lambda <r}b\left(
\lambda \right) \varphi _{\lambda }(x)$ which approximate the characteristic
function $\chi _{\Omega }(x)$ from above and below,
\begin{equation*}
A(x)\leq \chi _{\Omega }(x)\leq B(x),\;\;\;B(x)-A(x)\leq c\left( 1+r\mathrm{%
distance}\left\{ x,\partial \Omega \right\} \right) ^{-n}.
\end{equation*}%
In particular, the properties of the function $A(x)$ and of the measure $%
d\nu (x)$ give
\begin{gather*}
\int_{\Omega }d\nu (x)\geq \int_{\mathcal{M}}A(x)d\nu (x)=\int_{\mathcal{M}%
}A(x)dx \\
\geq \int_{\mathcal{M}}\chi _{\Omega }(x)dx-c\int_{\mathcal{M}}\left( 1+r%
\mathrm{distance}\left\{ x,\partial \Omega \right\} \right) ^{-n}dx.
\end{gather*}%
Similarly, by the properties of $B(x)$ and $d\nu (x)$,
\begin{gather*}
\int_{\Omega }d\nu (x)\leq \int_{\mathcal{M}}B(x)d\nu (x)=\int_{\mathcal{M}%
}B(x)dx \\
\leq \int_{\mathcal{M}}\chi _{\Omega }(x)dx+c\int_{\mathcal{M}}\left( 1+r%
\mathrm{distance}\left\{ x,\partial \Omega \right\} \right) ^{-n}dx.
\end{gather*}
\end{proof}

\begin{lem}
\label{L8}\ Assume that $d\nu (x)$\ is a probability measure on $\mathcal{M}$\ which
gives an exact quadrature for all eigenfunctions $\varphi _{\lambda }(x)$\
with eigenvalues $\lambda ^{2}<r^{2}$,
\begin{equation*}
\int_{\mathcal{M}}\varphi _{\lambda }(x)d\nu (x)=\int_{\mathcal{M}}\varphi
_{\lambda }(x)dx.
\end{equation*}%
If $1\leq q\leq +\infty $\ and $\alpha >d\left( 1-1/q\right) $, then there
exists $c$, which may depend on $q$, $\alpha $, $\mathcal{M}$, but is independent of $%
r $\ and $d\nu (x)$, such that
\begin{equation*}
\left\{ \int_{\mathcal{M}}\left\vert \int_{\mathcal{M}}B^{\alpha
}(x,y)d\nu (x)-1\right\vert ^{q}dy\right\} ^{1/q}\leq cr^{-\alpha }.
\end{equation*}
\end{lem}

\begin{proof}
Let $\chi \left( \lambda \right) $\ be an even smooth function on $-\infty
<\lambda <+\infty $\ with support in $1/2\leq \left\vert \lambda \right\vert
\leq 2$\ with the property that $\sum_{j=-\infty }^{+\infty }\chi \left(
2^{-j}\lambda \right) =1$ for every $\lambda \neq 0$. Also, let
\begin{equation*}
P^{\alpha }(s,x,y)=\sum_{\lambda }\chi \left( \lambda /s\right) \left(
1+\lambda ^{2}\right) ^{-\alpha /2}\varphi _{\lambda }(x)\overline{\varphi
_{\lambda }(y)}.
\end{equation*}%
Hence, $B^{\alpha }(x,y)=1+\sum_{j=-\infty }^{+\infty }P^{\alpha
}(2^{j},x,y)$. Since $d\nu (x)$ annihilates all eigenfunctions with $%
0<\lambda <r$, it also annihilates all $P^{\alpha }(2^{j},x,y)$ with $%
2^{j}\leq r/2$ and this gives
\begin{equation*}
\int_{\mathcal{M}}B^{\alpha }(x,y)d\nu (x)-1=\int_{\mathcal{M}}\left(
\sum_{2^{j}>r/2}P^{\alpha }(2^{j},x,y)\right) d\nu (x).
\end{equation*}%
When $q=1$, by Lemma \ref{L6},
\begin{gather*}
\int_{\mathcal{M}}\left\vert \int_{\mathcal{M}}P^{\alpha }(s,x,y)d\nu
(x)\right\vert dy \\
\leq cs^{d-\alpha }\int_{\mathcal{M}}\int_{\mathcal{M}}\left(
1+s\left\vert x-y\right\vert \right) ^{-n}d\nu (x)dy \\
\leq cs^{-\alpha }\sup_{x\in \mathcal{M}}\left\{ \int_{\mathcal{M}%
}s^{d}\left( 1+s\left\vert x-y\right\vert \right) ^{-n}dy\right\} \leq
cs^{-\alpha }.
\end{gather*}%
When $q=+\infty $ and $s\geq r$ and $n>d$, by Lemma \ref{L6} and Lemma \ref%
{L7},
\begin{gather*}
\sup_{y\in \mathcal{M}}\left\{ \left\vert \int_{\mathcal{M}}P^{\alpha
}(s,x,y)d\nu (x)\right\vert \right\} \\
\leq cs^{d-\alpha }\sup_{y\in \mathcal{M}}\left\{ \int_{\mathcal{M}}\left(
1+s\left\vert x-y\right\vert \right) ^{-n}d\nu (x)\right\} \\
\leq cs^{d-\alpha }\sup_{y\in \mathcal{M}}\left\{ \int_{\left\{ \left\vert
x-y\right\vert \leq 1/r\right\} }d\nu (x)\right\} \\
+cs^{d-\alpha }\sup_{y\in \mathcal{M}}\left\{ \sum_{j=0}^{+\infty }\left(
2^{j}s/r\right) ^{-n}\int_{\left\{ \left\vert x-y\right\vert \leq
2^{j}/r\right\} }d\nu (x)\right\} \\
\leq cs^{d-\alpha }r^{-d}+cs^{d-\alpha -n}r^{n-d}\leq cs^{d-\alpha }r^{-d}.
\end{gather*}%
Hence, when $s\geq r$ and $1<q<+\infty $, by interpolation between $1$ and $%
+\infty $,
\begin{gather*}
\left\{ \int_{\mathcal{M}}\left\vert \int_{\mathcal{M}}P^{\alpha
}(s,x,y)d\nu (x)\right\vert ^{q}dy\right\} ^{1/q} \\
\leq \sup_{y\in \mathcal{M}}\left\{ \left\vert \int_{\mathcal{M}}P^{\alpha
}(s,x,y)d\nu (x)\right\vert \right\} ^{(q-1)/q}\left\{ \int_{\mathcal{M}%
}\left\vert \int_{\mathcal{M}}P^{\alpha }(s,x,y)d\nu (x)\right\vert
dy\right\} ^{1/q} \\
\leq cs^{d(1-1/q)-\alpha }r^{-d(1-1/q)}.
\end{gather*}%
When $\alpha >d(1-1/q)$ these estimates sum to
\begin{gather*}
\left\{ \int_{\mathcal{M}}\left\vert \int_{\mathcal{M}}B^{\alpha
}(x,y)d\nu (x)-1\right\vert ^{q}dy\right\} \\
\leq \sum_{2^{j}>r/2}\left\{ \int_{\mathcal{M}}\left\vert \int_{\mathcal{M%
}}P^{\alpha }(2^{j},x,y)d\nu (x)\right\vert ^{q}dy\right\} ^{1/q} \\
\leq cr^{-d(1-1/q)}\sum_{2^{j}>r/2}2^{j\left( d(1-1/q)-\alpha \right) }\leq
cr^{-\alpha }.
\end{gather*}
\end{proof}

Finally, the existence of exact quadrature rules associated to any system of
continuous functions is a simple result in functional analysis, or in convex
geometry. See Theorem 3.1.1 in \cite{Shapiro}, or \cite{SZ}, or \cite{BFS}
for explicit constructions on spheres.

\begin{lem}
\label{L9}\ Given any number $\varphi _{1}(x)$, $\varphi _{2}(x)$,..., $%
\varphi _{n}(x)$\ of continuous functions on $\mathcal{M}$, there exist points $%
\left\{ z_{j}\right\} _{j=1}^{N}$\ in $\mathcal{M}$\ and positive weights $\left\{
\omega _{j}\right\} _{j=1}^{N}$\ with $\sum_{j=1}^{N}\omega _{j}=1$, such
that for every $\varphi _{i}(x)$,
\begin{equation*}
\int_{\mathcal{M}}\varphi _{i}(x)dx=\sum_{j=1}^{N}\omega _{j}\varphi
_{i}\left( z_{j}\right) .
\end{equation*}%
If the functions $\varphi _{j}(x)$\ are real one can choose $N\leq n+1$, and
if these functions are complex one can choose $N\leq 2n+1$.
\end{lem}

\begin{proof}
Define
\begin{gather*}
\Phi (x)=\left( \varphi _{1}(x),\varphi _{2}(x),...,\varphi _{n}(x)\right) ,
\\
E=\int_{\mathcal{M}}\Phi (x)dx=\left( \int_{\mathcal{M}}\varphi
_{1}(x)dx,\int_{\mathcal{M}}\varphi _{2}(x)dx,...,\int_{\mathcal{M}%
}\varphi _{n}(x)dx\right) .
\end{gather*}%
If all functions $\varphi _{i}(x)$ are real valued, then $\Phi (x)$ and $E$
are vectors in $\mathbb{R}^{n}$. If the $\varphi _{i}(x)$ are complex, then $%
\Phi (x)$ and $E$ can be seen as vectors in $\mathbb{R}^{2n}$. Moreover, $E$
is in the convex hull of the vectors $\Phi (x)$ with $x\in \mathcal{M}$. It
then follows from Caratheodory's theorem that $E$ is also a convex
combination of at most $n+1$ vectors $\Phi (x)$ in the real case, or $2n+1$
in the complex case, $E=\sum_{j=1}^{N}\omega _{j}\Phi \left( z_{j}\right) $
with $\omega _{j}\geq 0$ and $\sum_{j=1}^{N}\omega _{j}=1$.
\end{proof}

The following is Result (2) in the Introduction.

\begin{thm}
\label{T10}Assume that the probability measure $d\nu (x)$\ on $\mathcal{M}$\ gives an
exact quadrature for all eigenfunctions $\varphi _{\lambda }(x)$\ with
eigenvalues $\lambda ^{2}<r^{2}$,
\begin{equation*}
\int_{\mathcal{M}}\varphi _{\lambda }(x)d\nu (x)=\int_{\mathcal{M}}\varphi
_{\lambda }(x)dx.
\end{equation*}%
Then, for some constant $c$\ independent of $d\nu (x)$\ and $r$\ and for
every function $f(x)$\ in $W^{\alpha ,p}\left( \mathcal{M}\right) $\ with $%
1\leq p\leq +\infty $\ and $\alpha >d/p$,
\begin{equation*}
\left\vert \int_{\mathcal{M}}f(x)d\nu (x)-\int_{\mathcal{M}%
}f(x)dx\right\vert \leq cr^{-\alpha }\left\Vert f\right\Vert _{\mathbb{W}%
^{\alpha ,p}}.
\end{equation*}
\end{thm}

\begin{proof}
By Corollary \ref{C2} (1) with $d\mu (x)=d\nu (x)-dx$,
\begin{equation*}
\left\vert \int_{\mathcal{M}}f(x)d\mu (x)\right\vert \leq \left\{ \int_{%
\mathcal{M}}\left\vert \int_{\mathcal{M}}B^{\alpha }(x,y)d\nu
(x)-1\right\vert ^{q}dy\right\} ^{1/q}\left\Vert f\right\Vert _{\mathbb{W}%
^{\alpha ,p}}.
\end{equation*}%
By the assumption $\int_{\mathcal{M}}\varphi _{\lambda }(x)d\mu (x)=0$ for
every $\lambda <r$, and Lemma \ref{L8},
\begin{equation*}
\left\{ \int_{\mathcal{M}}\left\vert \int_{\mathcal{M}}B^{\alpha
}(x,y)d\nu (x)-1\right\vert ^{q}dy\right\} ^{1/q}\leq cr^{-\alpha }.
\end{equation*}
\end{proof}

The above theorem has as corollary Result (3) in the Introduction.

\begin{cor}
\label{C11} If $1\leq p\leq +\infty $\ and $\alpha >d/p$, then there exists $%
c>0$\ with the property that for every $N$\ there exist sequences of points $%
\left\{ z_{j}\right\} _{j=1}^{N}$\ and non negative weights $\left\{ \omega
_{j}\right\} _{j=1}^{N}$, such that for every function $f(x)$\ in $W^{\alpha
,p}\left( \mathcal{M}\right) $,
\begin{equation*}
\left\vert \sum_{j=1}^{N}\omega _{j}f\left( z_{j}\right) -\int_{\mathcal{M}%
}f(x)dx\right\vert \leq cN^{-\alpha /d}\left\Vert f\right\Vert _{\mathbb{W}%
^{\alpha ,p}}.
\end{equation*}
\end{cor}

\begin{proof}
By Weyl's estimates on the spectrum of an elliptic operator, see Theorem
17.5.3 in \cite{Hormander}, for a given $r$ there are approximately $cr^{d}$
eigenfunctions $\varphi _{\lambda }(x)$ with $\lambda <r$. The corollary then
follows from Lemma \ref{L9} and Theorem \ref{T10} with $r=N^{1/d}$.
\end{proof}

This corollary for the sphere is contained in \cite{SMS}. Finally, easy
examples show that the above estimates for the error in approximate
quadrature are, up to constants, best possible. Again, see \cite{HS1}
for the case of the sphere. In particular, the following is
Result (4) in the Introduction.

\begin{thm}
\label{T12}\textit{\ }For every $1\leq p\leq +\infty $\ and $\alpha >0$\
there exists $c>0$\ with the following property: For every distribution of
points $\left\{ z_{j}\right\} _{j=1}^{N}$\ there exists a function $f(x)$\
in $W^{\alpha ,p}\left( \mathcal{M}\right) $\ which vanishes in a
neighborhood of these points and satisfies
\begin{equation*}
\left\Vert f\right\Vert _{\mathbb{W}^{\alpha ,p}}\leq cN^{\alpha
/d},\;\;\;\int_{\mathcal{M}}f(x)dx=1.
\end{equation*}%
As a consequence, for every distribution of points $\left\{ z_{j}\right\}
_{j=1}^{N}$\ and complex weights $\left\{ \omega _{j}\right\} _{j=1}^{N}$,
there exists a function $f(x)$\ with
\begin{equation*}
\left\vert \sum_{j=1}^{N}\omega _{j}f\left( z_{j}\right) -\int_{\mathcal{M}%
}f(x)dx\right\vert \geq cN^{-\alpha /d}\left\Vert f\right\Vert _{\mathbb{W}%
^{\alpha ,p}}.
\end{equation*}
\end{thm}

\begin{proof}
If $\varepsilon $ is small, then one can choose $2N$ disjoint balls in $%
\mathcal{M}$ with diameters $\varepsilon N^{-1/d}$ and, given $N$ points $%
\left\{ z_{j}\right\} $, at least $N$ balls have no points inside. On each
empty ball construct a bump function $\psi _{j}(x)$ supported on it with
\begin{equation*}
\left\Vert \psi _{j}\right\Vert _{\mathbb{W}^{\alpha ,p}}\leq cN^{\alpha
/d-1/p},\;\;\;\int_{\mathcal{M}}\psi _{j}(x)dx=N^{-1}.
\end{equation*}%
Define $f(x)=\sum \psi _{j}(x)$. Then,
\begin{equation*}
\left\Vert f\right\Vert _{\mathbb{W}^{\alpha ,p}}\leq cN^{\alpha
/d},\;\;\;\int_{\mathcal{M}}f(x)dx=1.
\end{equation*}%
The estimate of the $\mathbb{L}^{p}\left( \mathcal{M}\right) $ norms of $%
\left( I+\Delta \right) ^{\alpha /2}\psi _{j}(x)$ and $\left( I+\Delta
\right) ^{\alpha /2}f(x)$ when $\alpha /2$ is an integer follows from the
fact that $\left( I+\Delta \right) ^{\alpha /2}$ is a differential operator
and the terms $\left( I+\Delta \right) ^{\alpha /2}\psi _{j}(x)$ have
disjoint supports. When $\alpha /2$ is not an integer the estimate follows
by complex interpolation. Anyhow, the case $p=2$ is elementary. If $\delta
>1 $ and $\alpha \delta $ is an integer,
\begin{gather*}
\left\Vert f\right\Vert _{\mathbb{W}^{\alpha ,2}}=\left\{ \sum_{\lambda
}\left( 1+\lambda ^{2}\right) ^{\alpha }\left\vert \mathcal{F}f(\lambda
)\right\vert ^{2}\right\} ^{1/2} \\
\leq \left\{ \sum_{\lambda }\left\vert \mathcal{F}f(\lambda )\right\vert
^{2}\right\} ^{\left( 1-1/\delta \right) /2}\left\{ \sum_{\lambda }\left(
1+\lambda ^{2}\right) ^{\alpha \delta }\left\vert \mathcal{F}f(\lambda
)\right\vert ^{2}\right\} ^{1/2\delta }\leq cN^{\alpha /d}.
\end{gather*}
\end{proof}

\section{Further results}

The following is Result (5) in the Introduction and it states that a
quadrature rule which gives an optimal error in the Sobolev space $\mathbb{W}%
^{\alpha ,2}\left( \mathcal{M}\right) $ is also optimal in all spaces $%
\mathbb{W}^{\beta ,2}\left( \mathcal{M}\right) $ with $d/2<\beta <\alpha $.

\begin{thm}
\label{T13}\ If $d\nu (x)$\ is a probability measure on $\mathcal{M}$, then the norm
of the measure $d\nu (x)-dx$\ as a linear functional on $W^{\alpha ,2}\left(
\mathcal{M}\right) $\ decreases as $\alpha $\ increases. Moreover, if the
norm of $d\nu (x)-dx$\ on $W^{\alpha ,2}\left( \mathcal{M}\right) $\ is $%
r^{-\alpha }$,
\begin{equation*}
\left\vert \int_{\mathcal{M}}f(x)d\nu (x)-\int_{\mathcal{M}%
}f(x)dx\right\vert \leq r^{-\alpha }\left\Vert f\right\Vert _{\mathbb{W}%
^{\alpha ,2}},
\end{equation*}%
then for every $d/2<\beta <\alpha $\ there exists a constant $c$\ which may
depend on $\alpha $, $\beta $, $\mathcal{M}$, but is independent of $r$\ and $d\nu (x)$%
, such that
\begin{equation*}
\left\vert \int_{\mathcal{M}}f(x)d\nu (x)-\int_{\mathcal{M}%
}f(x)dx\right\vert \leq cr^{-\beta }\left\Vert f\right\Vert _{\mathbb{W}%
^{\beta ,2}}.
\end{equation*}
\end{thm}

\begin{proof}
By Corollary \ref{C2} (2) the norm of the measure $d\nu (x)-dx$ as a linear
functional on $\mathbb{W}^{\alpha ,2}\left( \mathcal{M}\right) $ is
\begin{equation*}
\left\{ \int_{\mathcal{M}}\int_{\mathcal{M}}B^{2\alpha }\left( x,y\right)
d\nu (x)d\nu (y)-1\right\} ^{1/2}=\left\{ \sum_{\lambda >0}\left( 1+\lambda
^{2}\right) ^{-\alpha }\left\vert \mathcal{F}\nu \left( \lambda \right)
\right\vert ^{2}\right\} ^{1/2}.
\end{equation*}%
Since $\left( 1+\lambda ^{2}\right) ^{-\alpha }\leq \left( 1+\lambda
^{2}\right) ^{-\beta }$ when $\beta <\alpha $, it follows that this norm is
a decreasing function of $\alpha $. Write $d\nu (x)-dx=d\mu (x)$. By Lemma %
\ref{L4} (1), the norm of the functional $\int_{\mathcal{M}}f(x)d\mu (x)$
on $\mathbb{W}^{\alpha ,2}\left( \mathcal{M}\right) $ can be written as
\begin{gather*}
\left\{ \int_{\mathcal{M}}\int_{\mathcal{M}}B^{2\alpha }\left( x,y\right)
d\mu (x)\overline{d\mu (y)}\right\} ^{1/2} \\
=\left\{ \Gamma \left( \alpha \right) ^{-1}\int_{0}^{+\infty }t^{\alpha
-1}\exp \left( -t\right) \left( \int_{\mathcal{M}}\int_{\mathcal{M}%
}W\left( t,x,y\right) d\mu (x)\overline{d\mu (y)}\right) dt\right\} ^{1/2}.
\end{gather*}%
Assuming that this norm is $r^{-\alpha }$, one has to show that the
corresponding expression with $\beta $ instead of $\alpha $ is at most $%
cr^{-\beta }$. Since $\beta <\alpha $, the integral over $1\leq t<+\infty $
satisfies the estimate
\begin{gather*}
\int_{1}^{+\infty }t^{\beta -1}\exp \left( -t\right) \left( \int_{\mathcal{%
M}}\int_{\mathcal{M}}W\left( t,x,y\right) d\mu (x)\overline{d\mu (y)}%
\right) dt \\
\leq \int_{1}^{+\infty }t^{\alpha -1}\exp \left( -t\right) \left( \int_{%
\mathcal{M}}\int_{\mathcal{M}}W\left( t,x,y\right) d\mu (x)\overline{d\mu
(y)}\right) dt \\
\leq \Gamma \left( \alpha \right) r^{-2\alpha }.
\end{gather*}%
Similarly, since $\beta <\alpha $ the integral over $r^{-2}\leq t\leq 1$
satisfies the estimate
\begin{gather*}
\int_{r^{-2}}^{1}t^{\beta -1}\exp \left( -t\right) \left( \int_{\mathcal{M}%
}\int_{\mathcal{M}}W\left( t,x,y\right) d\mu (x)\overline{d\mu (y)}\right)
dt \\
\leq r^{2\alpha -2\beta }\int_{r^{-2}}^{1}t^{\alpha -1}\exp \left(
-t\right) \left( \int_{\mathcal{M}}\int_{\mathcal{M}}W\left( t,x,y\right)
d\mu (x)\overline{d\mu (y)}\right) dt \\
\leq \Gamma \left( \alpha \right) r^{-2\beta }.
\end{gather*}%
Finally, by the Gaussian estimate on the heat kernel in the proof of Lemma %
\ref{L3}, if $0<t<r^{-2}$ then
\begin{equation*}
t^{d/2}W\left( t,x,y\right) \leq cr^{-d}W\left( r^{-2},x,y\right) .
\end{equation*}%
It then follows that if $\beta >d/2$ the integral over $0\leq t\leq r^{-2}$
satisfies the estimate
\begin{gather*}
\int_{0}^{r^{-2}}t^{\beta -1}\exp \left( -t\right) \left( \int_{\mathcal{M}%
}\int_{\mathcal{M}}W\left( t,x,y\right) d\mu (x)\overline{d\mu (y)}\right)
dt \\
\leq cr^{-2\beta }\int_{\mathcal{M}}\int_{\mathcal{M}}W\left(
r^{-2},x,y\right) d\left\vert \mu \right\vert (x)d\left\vert \mu \right\vert
(y).
\end{gather*}%
It remains to show that the last double integral is uniformly bounded in $r$%
. Since $d\left\vert \mu \right\vert (x)=d\nu (x)+dx$ and since $\int_{%
\mathcal{M}}W\left( r^{-2},x,y\right) dx=1$, replacing $d\left\vert \mu
\right\vert (x)$ with $d\mu (x)$ it suffices to show that
\begin{equation*}
\int_{\mathcal{M}}\int_{\mathcal{M}}W\left( r^{-2},x,y\right) d\mu (x)d\mu
(y)\leq c.
\end{equation*}%
By the assumption on $d\mu (x)$ and the eigenfunction expansion of $W\left(
r^{-2},x,y\right) $,
\begin{gather*}
\int_{\mathcal{M}}\int_{\mathcal{M}}W\left( r^{-2},x,y\right) d\mu (x)d\mu
(y) \\
\leq r^{-\alpha }\left\Vert \int_{\mathcal{M}}W\left( r^{-2},x,y\right)
d\mu (y)\right\Vert _{\mathbb{W}^{\alpha ,2}} \\
=r^{-\alpha }\left\{ \sum_{\lambda }\left( 1+\lambda ^{2}\right) ^{\alpha
}\exp \left( -\left( \lambda /r\right) ^{2}\right) \left\vert \mathcal{F}\mu
(\lambda )\right\vert ^{2}\right\} ^{1/2} \\
\leq r^{-\alpha }\left\{ \sum_{\lambda }\left( 1+\lambda ^{2}\right)
^{-\alpha }\left\vert \mathcal{F}\mu (\lambda )\right\vert ^{2}\right\}
^{1/2}\sup_{\lambda }\left\{ \left( 1+\lambda ^{2}\right) ^{\alpha }\exp
\left( -\left( \lambda /r\right) ^{2}/2\right) \right\} .
\end{gather*}%
Finally, the last sum with $\left\{ \mathcal{F}\mu (\lambda )\right\} $ is
the norm of the measure $d\mu (x)$ as functional on $\mathbb{W}^{\alpha
,2}\left( \mathcal{M}\right) $, hence by assumption it is $r^{-\alpha }$,
and the last supremum is dominated by $r^{2\alpha }$.
\end{proof}

As we said, the above result is only one way, from $\alpha $ to $\beta
<\alpha $. If the norm of $d\nu (x)-dx$\ on $\mathbb{W}^{\alpha ,p}\left(
\mathcal{M}\right) $ is $r^{-\alpha }$ and if $\beta >\alpha $, then one
cannot conclude that the norm of $d\nu (x)-dx$\ on $\mathbb{W}^{\beta
,p}\left( \mathcal{M}\right) $ is at most $cr^{-\beta }$. As a
counterexample, it suffices to perturb a good quadrature rule with nodes $%
\left\{ z_{j}\right\} _{j=1}^{N}$\ and weights $\left\{ \omega _{j}\right\}
_{j=1}^{N}$ by moving the last point $z_{N}$ into a new point $t_{N}$, so
that the new quadrature differs from the old one by the quantity $\omega
_{N}\left\vert f\left( z_{N}\right) -f\left( t_{N}\right) \right\vert $.
If
$\alpha>d/p+1$ then the function $f$ is differentiable and
$\omega _{N}\left\vert f\left( z_{N}\right) -f\left(
t_{N}\right) \right\vert \approx \omega _{N}\left\vert z_{N}-t_{N}\right\vert.$
Then, by choosing $\left\vert z_{N}-t_{N}\right\vert
=r^{-\alpha }/\omega _{N}$ one obtains a quadrature rule which gives
an error $\approx r^{-\alpha}$ in all spaces $\mathbb{W}^{\beta
,p}\left( \mathcal{M}\right) $ with $\beta>\alpha$. The counterexample when
 $d/p<\alpha \leq d/p+1$ is slightly more complicated but similar.

In all the above results, the accuracy in a quadrature rule has been estimated
in terms of the energy of a measure. It is also possible to estimate this
accuracy in terms of a geometric discrepancy. The Bessel kernel can be
decomposed as superposition of characteristic functions,
\begin{equation*}
B^{\alpha }(x,y)=\int_{0}^{+\infty }\chi _{\left\{ B^{\alpha }\left(
x,y\right) >t\right\} }(x)dt.
\end{equation*}

If $1\leq p,q\leq +\infty $ and $1/p+1/q=1$, by Corollary \ref{C2} and
Minkowski inequality, the following Koksma Hlawka type inequality holds:
\begin{gather*}
\left\vert \int_{\mathcal{M}}f(x)d\mu (x)\right\vert \leq \left\Vert
f\right\Vert _{\mathbb{W}^{\alpha ,p}}\left\{ \int_{\mathcal{M}}\left\vert
\int_{\mathcal{M}}B^{\alpha }(x,y)d\mu (x)\right\vert ^{q}dy\right\} ^{1/q}
\\
\leq \left\Vert f\right\Vert _{\mathbb{W}^{\alpha ,p}}\int_{0}^{+\infty
}\left\{ \int_{\mathcal{M}}\left\vert \int_{\mathcal{M}}\chi _{\left\{
B^{\alpha }\left( x,y\right) >t\right\} }(x)d\mu (x)\right\vert
^{q}dy\right\} ^{1/q}dt.
\end{gather*}

The quantity $\left\vert \int_{\mathcal{M}}\chi _{\left\{ B^{\alpha }\left(
x,y\right) >t\right\} }(x)d\mu (x)\right\vert $ is the discrepancy of the
measure $d\mu (x)$ with respect to the level sets $\left\{ B^{\alpha }\left(
x,y\right) >t\right\} $. It can be proved that, for specific measures and at
least in the range $0<\alpha <1$, the above estimates are sharp and they can
lead to optimal quadrature rules. In particular, the following is Result (6)
in the Introduction.

\begin{thm}
\label{T14}\textit{\ }Denote by $\delta (t)$\ the diameter of the level sets
of the Bessel kernel $\left\{ B^{\alpha }\left( x,y\right) >t\right\} $\ and
assume that there exists $r\geq 1$\ such that the discrepancy of the measure
$d\mu (x)$\ with respect to $\left\{ B^{\alpha }\left( x,y\right) >t\right\}
$\ satisfies the estimates
\begin{equation*}
\left\vert \int_{\mathcal{M}}\chi _{\left\{ B^{\alpha }\left( x,y\right)
>t\right\} }(x)d\mu (x)\right\vert \leq \left\{
\begin{array}{l}
r^{-d}\;\;\;\text{\textit{if }}\delta (t)\leq 1/r\text{\textit{,}} \\
r^{-1}\delta (t)^{d-1}\;\;\;\text{\textit{if }}\delta (t)\geq 1/r\text{%
\textit{.}}%
\end{array}%
\right.
\end{equation*}%
Also assume that $1\leq p\leq +\infty $\ and $\alpha >d/p$. Then there
exists a constant $c$, which may depend on $\alpha $\ and $p$\ and on the
total variation of the measure $\left\vert \mu \right\vert \left( \mathcal{M}%
\right) $, but is independent of $r$, such that
\begin{equation*}
\left\vert \int_{\mathcal{M}}f(x)d\mu (x)\right\vert \leq \left\{
\begin{array}{l}
cr^{-\alpha }\left\Vert f\right\Vert _{\mathbb{W}^{\alpha ,p}}\;\;\;\text{%
\textit{if }}0<\alpha <1\text{\textit{,}} \\
cr^{-1}\log (1+r)\left\Vert f\right\Vert _{\mathbb{W}^{\alpha ,p}}\;\;\;%
\text{\textit{if }}\alpha =1\text{\textit{,}} \\
cr^{-1}\left\Vert f\right\Vert _{\mathbb{W}^{\alpha ,p}}\;\;\;\text{\textit{%
if }}\alpha >1\text{\textit{.}}%
\end{array}%
\right.
\end{equation*}
\end{thm}

\begin{proof}
Observe that the above hypotheses on the discrepancy match the estimates in
Lemma \ref{L7}. Indeed, by this lemma, the measures $d\nu (x)$ which give
exact quadrature for eigenfunctions with eigenvalues $\lambda ^{2}<r^{2}$
have discrepancy
\begin{equation*}
\left\vert \int_{\left\{ \left\vert x-y\right\vert \leq s\right\} }d\nu
(x)-\int_{\left\{ \left\vert x-y\right\vert \leq s\right\} }dx\right\vert
\leq \left\{
\begin{array}{l}
cr^{-d}\;\;\;\text{if }s\leq 1/r\text{,} \\
cr^{-1}s^{d-1}\;\;\;\text{if }s\geq 1/r\text{.}%
\end{array}%
\right.
\end{equation*}%
Actually, these estimates hold not only for balls $\left\{ \left\vert
x-y\right\vert \leq s\right\} $, but also for sets with boundaries with
finite $d-1$ dimensional Minkowski measure, such as the level sets $\left\{
B^{\alpha }\left( x,y\right) >t\right\} $. Also observe that these estimates
are natural, since the discrepancy of large sets is qualitatively different
from the one of small sets. If $1\leq p,q\leq +\infty $ and $1/p+1/q=1$, by
Corollary \ref{C2} and Minkowski inequality,
\begin{equation*}
\left\vert \int_{\mathcal{M}}f(x)d\mu (x)\right\vert \leq \left\Vert
f\right\Vert _{\mathbb{W}^{\alpha ,p}}\int_{0}^{+\infty }\left\{ \int_{%
\mathcal{M}}\left\vert \int_{\mathcal{M}}\chi _{\left\{ B^{\alpha }\left(
x,y\right) >t\right\} }(x)d\mu (x)\right\vert ^{q}dy\right\} ^{1/q}dt.
\end{equation*}%
By Lemma \ref{L4}, when $0<\alpha <d$ then $B^{\alpha }\left( x,y\right)
\approx \left\vert x-y\right\vert ^{\alpha -d}$, the level sets $\left\{
B^{\alpha }\left( x,y\right) >t\right\} $ have diameters $\delta (t)\approx
\min \left\{ 1,t^{1/(\alpha -d)}\right\} $ and the boundaries $\left\{
B^{\alpha }\left( x,y\right) =t\right\} $ have surface measure of the order
of $\delta (t)^{d-1}\approx \min \left\{ 1,t^{(d-1)/(\alpha -d)}\right\} $.
Hence the estimate of the discrepancy of small level sets with $t\geq
r^{d-\alpha }$ gives
\begin{gather*}
\left\{ \int_{\mathcal{M}}\left\vert \int_{\mathcal{M}}\chi _{\left\{
B^{\alpha }\left( x,y\right) >t\right\} }(x)d\mu (x)\right\vert
^{q}dy\right\} ^{1/q} \\
\leq \sup_{y\in \mathcal{M}}\left\{ \left\vert \int_{\mathcal{M}}\chi
_{\left\{ B^{\alpha }\left( x,y\right) >t\right\} }(x)d\mu (x)\right\vert
\right\} ^{(q-1)/q}\left\{ \int_{\mathcal{M}}\int_{\mathcal{M}}\chi
_{\left\{ B^{\alpha }\left( x,y\right) >t\right\} }(x)d\left\vert \mu
\right\vert (x)dy\right\} ^{1/q} \\
\leq \sup_{y\in \mathcal{M}}\left\{ \left\vert \int_{\mathcal{M}}\chi
_{\left\{ B^{\alpha }\left( x,y\right) >t\right\} }(x)d\mu (x)\right\vert
\right\} ^{(q-1)/q}\left\{ c\left\vert \mu \right\vert (\mathcal{M}%
)t^{d/(\alpha -d)}\right\} ^{1/q} \\
\leq cr^{-d(q-1)/q}t^{d/q(\alpha -d)}.
\end{gather*}%
Hence, if $\alpha >d/p$ the integral over $r^{d-\alpha }\leq t<+\infty $
satisfies the inequality
\begin{gather*}
\int_{r^{d-\alpha }}^{+\infty }\left\{ \int_{\mathcal{M}}\left\vert \int_{%
\mathcal{M}}\chi _{\left\{ B^{\alpha }\left( x,y\right) >t\right\} }(x)d\mu
(x)\right\vert ^{q}dy\right\} ^{1/q}dt \\
\leq cr^{-d(q-1)/q}\int_{r^{d-\alpha }}^{+\infty }t^{d/q(\alpha -d)}dt\leq
cr^{-\alpha }.
\end{gather*}%
Similarly, the integral over $0\leq t\leq r^{d-\alpha }$, that is the
discrepancy of large level sets, satisfies the inequality
\begin{gather*}
\int_{0}^{r^{d-\alpha }}\left\{ \int_{\mathcal{M}}\left\vert \int_{%
\mathcal{M}}\chi _{\left\{ B^{\alpha }\left( x,y\right) >t\right\} }(x)d\mu
(x)\right\vert ^{q}dy\right\} ^{1/q}dt \\
\leq r^{-1}\int_{0}^{r^{d-\alpha }}\min \left\{ 1,t^{(d-1)/(\alpha
-d)}\right\} dt\leq \left\{
\begin{array}{l}
cr^{-\alpha }\;\;\;\text{if }0<\alpha <1\text{,} \\
cr^{-1}\log (1+r)\;\;\;\text{if }\alpha =1\text{,} \\
cr^{-1}\;\;\;\text{if }\alpha >1\text{.}%
\end{array}%
\right.
\end{gather*}%
The proof in the case $\alpha =d$ is similar and it follows from the
estimate $B^{\alpha }(x,y)\approx -\log \left( \left\vert x-y\right\vert
\right) $. The proof in the case $\alpha >d$ is even simpler, since in this
case $B^{\alpha }(x,y)$ is bounded and it suffices to integrate on $0\leq
t\leq \sup_{x,y\in \mathcal{M}}B^{\alpha }\left( x,y\right) $ the inequality
$\left\vert \int_{\mathcal{M}}\chi _{\left\{ B^{\alpha }\left( x,y\right)
>t\right\} }(x)d\mu (x)\right\vert \leq cr^{-1}$.\
\end{proof}

In particular, it follows from Lemma \ref{L7}, Theorem \ref{T10}, Theorem %
\ref{T12}, that, at least in the range $0<\alpha <1$, Theorem \ref{T13}
gives an optimal quadrature. We conclude with a series of remarks.

\begin{rem}
\label{R15}As we said, the assumption $\alpha >d/2$ with $p=2$ in Theorem %
\ref{T5}, or $\alpha >d/p$ with $1\leq p\leq +\infty $ in Theorem \ref{T10},
guarantees the boundedness and continuity of $f\left( x\right) $, otherwise $%
f\left( z_{j}\right) $ may be not defined. This follows from the Sobolev
imbedding theorem. Indeed, the imbedding is an easy corollary of Lemma \ref%
{L4}. A function $f(x)$ is in the Sobolev space $\mathbb{W}^{\alpha
,p}\left( \mathcal{M}\right) $ if and only if there exists a function $g(x)$
in $\mathbb{L}^{p}\left( \mathcal{M}\right) $ with
\begin{equation*}
f(x)=\int_{\mathcal{M}}B^{\alpha }(x,y)g(y)dy.
\end{equation*}%
When $1\leq p,q\leq +\infty $, $1/p+1/q=1$, $d/p<\alpha <d$, then $B^{\alpha
}(x,y)\leq c\left\vert x-y\right\vert ^{\alpha -d}$ is in $\mathbb{L}%
^{q}\left( \mathcal{M}\right) $ and this implies that distributions in the
Sobolev space $\mathbb{W}^{\alpha ,p}\left( \mathcal{M}\right) $ with $%
\alpha >d/p$ are continuous functions. Indeed they are also H\"{o}lder
continuous of order $\alpha -d/p$.
\end{rem}

\begin{rem}
\label{R16} When the manifold is a Lie group or a homogeneous space, one can
restate Theorem \ref{T1} in terms of convolutions. In the particular case of
the torus $\mathbb{T}^{d}=\mathbb{R}^{d}/\mathbb{Z}^{d}$, let
\begin{equation*}
A(x)=\sum_{k\in \mathbb{Z}^{d}}\psi (k)\exp \left( 2\pi ikx\right)
,\;\;\;B(x)=\sum_{k\in \mathbb{Z}^{d}}\psi (k)^{-1}\exp \left( 2\pi
ikx\right) .
\end{equation*}%
Then, if $1\leq p,q,r\leq +\infty $\ with $1/p+1/q=1/r+1$,
\begin{gather*}
\left\{ \int_{\mathbb{T}^{d}}\left\vert \int_{\mathbb{T}^{d}}f\left(
x-y\right) d\mu (y)\right\vert ^{r}dx\right\} ^{1/r}=\left\{ \int_{\mathbb{T%
}^{d}}\left\vert B\ast A\ast f\ast \mu (x)\right\vert ^{r}dx\right\} ^{1/r}
\\
\leq \left\{ \int_{\mathbb{T}^{d}}\left\vert A\ast f(x)\right\vert
^{p}dx\right\} ^{1/p}\left\{ \int_{\mathbb{T}^{d}}\left\vert B\ast \mu
(x)\right\vert ^{q}dx\right\} ^{1/q}.
\end{gather*}

In the case of the sphere $\mathbb{S}^{d}=\left\{ x\in \mathbb{R}%
^{d+1},\;\left\vert x\right\vert =1\right\} $, let $\left\{ Z_{n}\left(
xy\right) \right\} $\ be the system of zonal spherical harmonics polynomials
and let
\begin{equation*}
A(xy)=\sum_{n=0}^{+\infty }\psi (n)Z_{n}\left( xy\right)
,\;\;\;B(xy)=\sum_{n=0}^{+\infty }\psi (n)^{-1}Z_{n}\left( xy\right) .
\end{equation*}%
Then, if $1\leq p,q\leq +\infty $\ with $1/p+1/q=1$,
\begin{gather*}
\left\vert \int_{\mathbb{S}^{d}}f\left( x\right) d\mu (x)\right\vert \\
\leq \left\{ \int_{\mathbb{S}^{d}}\left\vert \int_{\mathbb{S}%
^{d}}A(xy)f(y)dy\right\vert ^{p}dx\right\} ^{1/p}\left\{ \int_{\mathbb{S}%
^{d}}\left\vert \int_{\mathbb{S}^{d}}B(xy)d\mu (y)\right\vert
^{q}dx\right\} ^{1/q}.
\end{gather*}

Both results on the torus and the sphere follow from Young inequality for
convolutions.
\end{rem}

\begin{rem}
\label{R17} A result related to Theorem \ref{T1} is the following. Identify $%
\mathbb{T}^{d}$\ with the unit cube $\left\{ 0\leq x_{j}<1\right\} $\ and
denote by $\chi _{P(y)}(x)$\ the characteristic function of the
parallelepiped $P(y)=\left\{ 0\leq x_{j}<y_{j}\right\} $. Then define
\begin{gather*}
B(x)=\int_{\mathbb{T}^{d}}\chi _{P(y)}(x)dy-2^{-d}=\prod_{j=1}^{d}\left(
1-x_{j}\right) -2^{-d} \\
=\sum_{k\in \mathbb{Z}^{d}-\left\{ 0\right\} }\left( \left(
\prod_{k_{j}=0}2\right) \left( \prod_{k_{j}\neq 0}2\pi ik_{j}\right) \right)
^{-1}\exp \left( 2\pi ikx\right) .
\end{gather*}%
Also, define the differential integral operator
\begin{gather*}
A\ast f(x)=\sum_{k\neq 0}\left( \prod_{k_{j}=0}2\right) \left(
\prod_{k_{j}\neq 0}2\pi ik_{j}\right) \widehat{f}(k)\exp \left( 2\pi
ikx\right)  \\
=2^{d-1}\sum_{1\leq j\leq d}\int_{\mathbb{T}^{d-1}}\frac{\partial }{%
\partial x_{j}}f(x)\prod_{i\neq j}dx_{i}+2^{d-2}\sum_{1\leq i\neq j\leq
d}\int_{\mathbb{T}^{d-2}}\frac{\partial ^{2}}{\partial x_{i}\partial x_{j}}%
f(x)\prod_{h\neq i,j}dx_{h} \\
...+\frac{\partial ^{d}}{\partial x_{1}...\partial x_{d}}f(x).
\end{gather*}%
Observe that, as in Theorem \ref{T1}, the Fourier coefficients of the
distribution $A(x)$ and of the function $B(x)$ are one inverse to the other,
however here the Fourier coefficients are function of the lattice points $%
2\pi ik$, and not of the eigenvalues $4\pi ^{2}\left\vert k\right\vert ^{2}$%
. If $d\nu (x)=N^{-1}\sum_{j=1}^{N}\delta _{z_{j}}\left( x\right) $, and if
$1\leq p,q,r\leq +\infty $\ with $1/p+1/q=1/r+1$, then
\begin{gather*}
\left\{ \int_{\mathbb{T}^{d}}\left\vert N^{-1}\sum_{j=1}^{N}f\left(
x-z_{j}\right) -\int_{\mathbb{T}^{d}}f(y)dy\right\vert ^{r}dx\right\} ^{1/r}
\\
\leq \left\{ \int_{\mathbb{T}^{d}}\left\vert A\ast f(x)\right\vert
^{p}dx\right\} ^{1/p}\left\{ \int_{\mathbb{T}^{d}}\left\vert B\ast \nu
(x)\right\vert ^{q}dx\right\} ^{1/q}.
\end{gather*}%
The norm of $A\ast f(x)$ is dominated by an analogue of the Hardy Krause
variation,
\begin{gather*}
\left\{ \int_{\mathbb{T}^{d}}\left\vert A\ast f(x)\right\vert
^{p}dx\right\} ^{1/p} \\
\leq 2^{d-1}\sum_{1\leq j\leq d}\left\{ \int_{\mathbb{T}}\left\vert \int_{%
\mathbb{T}^{d-1}}\frac{\partial }{\partial x_{j}}f(x)\prod_{i=j}dx_{i}\right%
\vert ^{p}dx_{j}\right\} ^{1/p} \\
+2^{d-2}\sum_{1\leq i\neq j\leq d}\left\{ \int_{\mathbb{T}^{2}}\left\vert
\int_{\mathbb{T}^{d-2}}\frac{\partial ^{2}}{\partial x_{i}\partial x_{j}}%
f(x)\prod_{h\neq i,j}dx_{h}\right\vert ^{p}dx_{i}dx_{j}\right\} ^{1/p} \\
...+\left\{ \int_{\mathbb{T}^{d}}\left\vert \frac{\partial ^{d}}{\partial
x_{1}...\partial x_{d}}f(x)\right\vert ^{p}dx\right\} ^{1/p}.
\end{gather*}%
The norm of $B\ast \nu (x)$\ is dominated by the discrepancy of the points $%
\left\{ z_{j}\right\} _{j=1}^{N}$\ with respect to the family of boxes $P(y)$%
,
\begin{gather*}
\left\{ \int_{\mathbb{T}^{d}}\left\vert B\ast \nu (x)\right\vert
^{q}dx\right\} ^{1/q} \\
\leq \int_{\mathbb{T}^{d}}\left\{ \int_{\mathbb{T}^{d}}\left\vert
N^{-1}\sum_{j=1}^{N}\chi _{P(y)}\left( z_{j}+x\right)
-\prod_{j=1}^{d}y_{j}\right\vert ^{q}dx\right\} ^{1/q}dy.
\end{gather*}%
In particular, the case $p=1$ and $q=+\infty $ is an analogue of the Koksma
Hlawka inequality. See \cite{KN}. A generalization of this classical
inequality is contained in \cite{BCGT}.
\end{rem}

\begin{rem}
\textbf{\label{R18}} By Lemma \ref{L4} (1), the Bessel kernel $B^{\alpha
}(x,y)$\ with $\alpha >0$\ is a superposition of heat kernels $W\left(
t,x,y\right) $. Indeed, it is possible to state an analogue of Corollary %
\ref{C2} in terms of the heat kernel, without explicit mention of Bessel
potentials: If $\left\{ z_{j}\right\} _{j=1}^{N}$ is a sequence of points in
$\mathcal{M}$, if $\left\{ \omega _{j}\right\} _{j=1}^{N}$ are positive
weights with $\sum_{j}\omega _{j}=1$, and if $f(x)$ is a function in $%
\mathbb{W}^{\alpha ,p}\left( \mathcal{M}\right) $ with $\alpha >d/2$, then
\begin{gather*}
\left\vert \sum_{j=1}^{N}\omega _{j}f\left( z_{j}\right) -\int_{\mathcal{M}%
}f(x)dx\right\vert \\
\leq \left\{ \Gamma \left( \alpha \right) ^{-1}\int_{0}^{+\infty
}\left\vert \sum_{i=1}^{N}\sum_{j=1}^{N}\omega _{i}\omega _{j}W\left(
t,z_{i},z_{j}\right) -1\right\vert t^{\alpha -1}\exp \left( -t\right)
dt\right\} ^{1/2}\left\Vert f\right\Vert _{\mathbb{W}^{\alpha ,2}}.
\end{gather*}%
This suggests the following heuristic interpretation: Mathematically, a set
of points on a manifold is well-distributed if the associated Riemann sums
are close to the integrals. Physically, a set of points is well-distributed
if the heat, initially concentrated on them, in a short time diffuses
uniformly across the manifold.
\end{rem}

\begin{rem}
\label{R19} In order to minimize the errors in the numerical integration in
Corollary \ref{C2} (3), one has to minimize the energies
\begin{equation*}
\int_{\mathcal{M}}\int_{\mathcal{M}}B^{2\alpha }\left( x,y\right) d\nu
(x)d\nu (y),\;\;\;\sum_{i=1}^{N}\sum_{j=1}^{N}\omega _{i}\omega
_{j}B^{2\alpha }\left( z_{i},z_{j}\right) .
\end{equation*}%
These are analogous to the energy integrals in potential theory
\begin{equation*}
\int_{\mathcal{M}}\int_{\mathcal{M}}\left\vert x-y\right\vert
^{-\varepsilon }d\nu (x)d\nu (y).
\end{equation*}%
See \cite{HaSa}. When $d<\alpha <d+1$ the kernel $B^{2\alpha }\left(
x,y\right) $ is positive and bounded, with a maximum at $x=y$ and a spike $%
A-B\left\vert x-y\right\vert ^{2\alpha -d}$ when $x\rightarrow y$. In
particular, the gradient at $x=y$ is infinite. This implies that in order to
minimize the discrete energy $\sum_{i,j}\omega _{i}\omega _{j}B^{2\alpha
}\left( z_{i},z_{j}\right) $ the points $\left\{ z_{j}\right\} $ have to be
well separated. This suggests the following heuristic interpretation:
Mathematically, a set of points on a manifold is well-distributed if the
energy is minimal. Physically, a set of points, free to move and repelling
each other according to some law, is well-distributed when they reach an
equilibrium.
\end{rem}

\begin{rem}
\textbf{\label{R20}}It can be proved that if $2\alpha >d+2$ then
\begin{equation*}
\left\vert B^{2\alpha }\left( x,x\right) -B^{2\alpha }\left( x,y\right)
\right\vert \leq c\left\vert x-y\right\vert ^{2}.
\end{equation*}%
This estimate in the proof of Theorem \ref{T5} yields that for most choices
of sampling points $z_{j}\in U_{j}$,
\begin{equation*}
\left\vert \sum_{j=1}^{N}\omega _{j}f\left( z_{j}\right) -\int_{\mathcal{M}%
}f(x)dx\right\vert \leq c\max_{1\leq j\leq N}\left\{ \mathrm{diameter}%
\left( U_{j}\right) ^{d/2+1}\right\} \left\Vert f\right\Vert _{\mathbb{W}%
^{\alpha ,2}\left( \mathcal{M}\right) }.
\end{equation*}%
The same result holds if $2\alpha =d+2$, with a logarithmic transgression.
Observe that these estimates hold for most choices of sampling points, but
not for all choices. Indeed, if the manifold $\mathcal{M}$ is decomposed in
disjoint pieces $\mathcal{M}=U_{1}\cup U_{2}\cup ...\cup U_{N}$ with measure
$aN^{-1}\leq \left\vert U_{j}\right\vert =\omega _{j}\leq bN^{-1}$ and $%
\mathrm{diameter}\left( U_{j}\right) \leq cN^{-1/d}$, if $f(x)$ is a smooth
non constant function, and if the points $z_{j}\in U_{j}$ are the maxima of $%
f(x)$ in $U_{j}$, then $\sum_{j=1}^{N}\omega _{j}f\left( z_{j}\right) $ is
an upper sum of the integral $\int_{\mathcal{M}}f(x)dx$ and
\begin{equation*}
\sum_{j=1}^{N}\omega _{j}f\left( z_{j}\right) -\int_{\mathcal{M}%
}f(x)dx\geq cN^{-1/d}.
\end{equation*}
\end{rem}

\begin{rem}
\textbf{\label{R21}} Theorem \ref{T14} gives an estimate of the accuracy in
a quadrature rule in terms of the discrepancy of a measure with respect to
level sets of the Bessel kernel. The following argument shows that when the
manifold is a sphere, or a rank one compact symmetric space, then the level
sets of the heat kernel $\left\{ W\left( t,x,y\right) >s\right\} $, and
hence of the Bessel kernels $\left\{ B^{\alpha }\left( x,y\right) \leq
t\right\} $, are geodesic balls $\left\{ \left\vert x-y\right\vert \leq
r\right\} $. The Laplace operator on the sphere $\mathbb{S}^{d}$\ with
respect to a system of polar coordinates $x=\left( \vartheta ,\sigma \right)
$, with $0\leq \vartheta \leq \pi $\ the colatitude with respect to a given
pole and $\sigma \in \mathbb{S}^{d-1}$\ the longitude, is
\begin{equation*}
\Delta _{x}=\Delta _{\left( \vartheta ,\sigma \right) }=-\sin ^{1-d}\left(
\vartheta \right) \dfrac{\partial }{\partial \vartheta }\left( \sin
^{d-1}\left( \vartheta \right) \dfrac{\partial }{\partial \vartheta }\right)
+\Delta _{\sigma }.
\end{equation*}%
Let $u\left( t,x\right) $ be the solution of the Cauchy problem for the heat
equation
\begin{equation*}
\left\{
\begin{array}{l}
\dfrac{\partial }{\partial t}u\left( t,x\right) =-\Delta _{x}u\left(
t,x\right) , \\
u\left( 0,x\right) =f\left( x\right) .%
\end{array}%
\right.
\end{equation*}%
If $f\left( x\right) $ depends only on the colatitude $\vartheta $, if it is
an even function decreasing in $0<\vartheta <\pi $, then also $u\left(
t,x\right) $ depends only on the colatitude and it is an even function
decreasing in $0<\vartheta <\pi $. In order to prove this, set $%
u(t,x)=U(t,\vartheta )$, $f\left( x\right) =F\left( \vartheta \right) $, and
$\sin ^{d-1}\left( \vartheta \right) \partial U(t,\vartheta )/\partial
\vartheta =V(t,\vartheta )$. Then
\begin{gather*}
\left\{
\begin{array}{l}
\dfrac{\partial }{\partial \vartheta }\dfrac{\partial }{\partial t}%
U(t,\vartheta )=\dfrac{\partial }{\partial \vartheta }\left\{ \sin
^{1-d}\left( \vartheta \right) \dfrac{\partial }{\partial \vartheta }\left(
\sin ^{d-1}\left( \vartheta \right) \dfrac{\partial }{\partial \vartheta }%
U(t,\vartheta )\right) \right\} , \\
\dfrac{\partial }{\partial \vartheta }U(0,\vartheta )=\dfrac{\partial }{%
\partial \vartheta }F\left( \vartheta \right) ,%
\end{array}%
\right. \\
\left\{
\begin{array}{l}
\dfrac{\partial }{\partial t}V(t,\vartheta )=\dfrac{\partial ^{2}}{\partial
\vartheta ^{2}}V(t,\vartheta )+\left( 1-d\right) \dfrac{\cos (\vartheta )}{%
\sin (\vartheta )}\dfrac{\partial }{\partial \vartheta }V(t,\vartheta ), \\
V(0,\vartheta )=\sin ^{d-1}\left( \vartheta \right) \dfrac{\partial }{%
\partial \vartheta }F\left( \vartheta \right) , \\
V(t,0)=V(t,\pi )=0.%
\end{array}%
\right.
\end{gather*}%
If $F\left( \vartheta \right) $ is decreasing in $0<\vartheta <\pi $, then $%
V(0,\vartheta )\leq 0$ and, by the maximum principle, $V(t,\vartheta )\leq 0$%
, hence $U(t,\vartheta )$ is decreasing in $0<\vartheta <\pi $. In
particular, by considering a sequence of initial data $\left\{ f_{n}\left(
x\right) \right\} $ which depend only on the colatitude $\vartheta $, even
and decreasing in $0<\vartheta <\pi $, and which converge to the Dirac $%
\delta (x)$, one proves that the heat kernel $W\left( t,\cos \left(
\vartheta \right) \right) $ is decreasing in $0<\vartheta <\pi $. Since
Bessel kernels are superposition of heat kernels, they are also
superposition of spherical caps.
\end{rem}

\begin{rem}
\label{R22} In \cite{AK} and \cite{LPS} the discrepancy of orbits of
discrete subgroups of rotations of a sphere are studied. Let $\mathcal{G}$\
be a compact Lie group, $\mathcal{K}$\ a closed subgroup, $\mathcal{M}=%
\mathcal{G}/\mathcal{K}$\ a homogeneous space of dimension $d$. Also, let $%
\mathcal{H}$\ be a finitely generated free subgroup in $\mathcal{G}$\ and
assume that the action of $\mathcal{H}$\ on $\mathcal{M}$\ is free. Given a
positive integer $n$, let $\left\{ \sigma _{j}\right\} _{j=1}^{N}$ be an
ordering of the elements in $\mathcal{H}$ with length at most $n$ and for
every function $f(x)$\ on $\mathcal{M}$, define
\begin{equation*}
Tf(x)=N^{-1}\sum_{j=1}^{N}f\left( \sigma _{j}x\right) .
\end{equation*}%
This operator is self-adjoint and it has eigenvalues and eigenfunctions in $%
\mathbb{L}^{2}(\mathcal{M})$. Moreover, since the operators $T$ and $\Delta $
commute, they have a common orthonormal system of eigenfunctions, $\Delta
\varphi _{\lambda }(x)=\lambda ^{2}\varphi _{\lambda }(x)$ and $T\varphi
_{\lambda }(x)=T(\lambda )\varphi _{\lambda }(x)$. All eigenvalues of $T$
have modulus at most $1$ and indeed $1$ is an eigenvalue and the constants
are eigenfunctions. Assume that all non constant eigenfunctions have
eigenvalues much smaller than $1$. Then, if $\alpha >d/2$,
\begin{gather*}
\left\vert N^{-1}\sum_{j=1}^{N}f\left( \sigma _{j}x\right) -\int_{\mathcal{%
M}}f(x)dx\right\vert =\left\vert \sum\limits_{\lambda \neq 0}T(\lambda )%
\mathcal{F}f(\lambda )\varphi _{\lambda }(x)\right\vert  \\
\leq \left\{ \sup_{\lambda \neq 0}\left\{ \left\vert T(\lambda )\right\vert
\right\} \right\} \left\{ \sum\limits_{\lambda }\left( 1+\lambda
^{2}\right) ^{\alpha }\left\vert \mathcal{F}f(\lambda )\right\vert
^{2}\right\} ^{1/2}\left\{ \sum\limits_{\lambda }\left( 1+\lambda
^{2}\right) ^{-\alpha }\left\vert \varphi _{\lambda }(x)\right\vert
^{2}\right\} ^{1/2} \\
\leq c\left\{ \sup_{\lambda \neq 0}\left\{ \left\vert T(\lambda )\right\vert
\right\} \right\} \left\{ \int_{\mathcal{M}}\left\vert \left( I+\Delta
\right) ^{\alpha /2}f(x)\right\vert ^{2}dx\right\} ^{1/2}.
\end{gather*}%
The absolute convergence of the above series is consequence of the Sobolev's
imbeddings, or the Weyl's estimates for eigenfunctions. In particular, when $%
\mathcal{M}=SO(3)/SO(2)$\ is the two dimensional sphere and $\mathcal{H}$\
is the free group generated by rotations of angles $\arccos (-3/5)$\ around
orthogonal axes, it has been proved in \cite{LPS} that the eigenvalues of
the operator $T$ satisfy the Ramanujan bounds
\begin{equation*}
\sup_{\lambda \neq 0}\left\{ \left\vert T(\lambda )\right\vert \right\} \leq
cN^{-1/2}\log (N).
\end{equation*}%
Hence, for the sphere,
\begin{gather*}
\left\vert N^{-1}\sum_{j=1}^{N}f\left( \sigma _{j}x\right) -\int_{\mathcal{%
M}}f(x)dx\right\vert  \\
\leq cN^{-1/2}\log (N)\left\{ \int_{\mathcal{M}}\left\vert \left( I+\Delta
\right) ^{\alpha /2}f(x)\right\vert ^{2}dx\right\} ^{1/2}.
\end{gather*}%
All of this is essentially contained in \cite{LPS}. Although this bound $%
N^{-1/2}\log (N)$ is worse than the bound $N^{-\alpha /2}$ in Corollary \ref%
{C11}, the matrices $\left\{ \sigma _{j}\right\} $ have rational entries and
the sampling points $\left\{ \sigma _{j}x\right\} $ are completely explicit.
\end{rem}

%\bigskip
%\footnotesize
%\noindent\textit{Acknowledgments.}
%This research was partly supported by NSF (grant no. XXXX).

\end{document}